\documentclass[a4paper,11pt]{article}
\usepackage{makeidx}
\usepackage{graphicx}
\usepackage[english,activeacute]{babel}
\usepackage[latin1]{inputenc}
\usepackage{amsmath}
\usepackage{amsfonts}
\usepackage{amssymb,latexsym}
\usepackage{colortbl}
\usepackage{multicol}
\usepackage{hyperref}
\usepackage{epsfig}
\usepackage{mathrsfs}
\usepackage{subfig}

\newtheorem{theorem}{theorem}

\newtheorem{corollary}{corollary}

\newtheorem{lemma}{lemma}

\newtheorem{proposition}{proposition}
\newtheorem{remark}{remark}

\newenvironment{proof}[1][Proof]{\noindent\textbf{#1.} }{\ \rule{0.5em}{0.5em}}
\begin{document}

\title{On second order $q$-difference equations satisfied by
Al-Salam-Carlitz I-Sobolev type polynomials of higher order}
\author{Carlos Hermoso$^{1}$, Edmundo J. Huertas$^{1}$, Alberto Lastra$^{1}$%
, Anier Soria-Lorente$^{2,\dagger}$, \\
%EndAName
$^{1}$Departamento de Física y Matemáticas, Universidad de Alcalá,\\
Ctra. Madrid-Barcelona, Km. 33,600.\\
Facultad de Ciencias, 28805 - Alcalá de Henares, Madrid, Spain.\\
carlos.hermoso@uah.es \& edmundo.huertas@uah.es \& alberto.lastra@uah.es\\
$^{2}$ Facultad de Ciencias Técnicas, Universidad de Granma,\\
Km. 17,5 de la carretera de Bayamo-Manzanillo, 85100 Bayamo, Cuba.\\
asorial@udg.co.cu, asorial1983@gmail.com\\
$^{\dagger }$Corresponding author. }
\date{\emph{(\today)}}
\maketitle

\begin{abstract}
This contribution deals with the sequence $\{\mathbb{U}_{n}^{(a)}(x;q,j)%
\}_{n\geq 0}$ of monic polynomials in $x$, orthogonal with respect to a
Sobolev-type inner product related to the Al-Salam--Carlitz I orthogonal
polynomials, and involving an arbitrary number $j$ of $q$-derivatives on the
two boundaries of the corresponding orthogonality interval, for some fixed
real number $q\in(0,1)$. We provide several versions of the corresponding
connection formulas, ladder operators, and several versions of the second
order $q$-difference equations satisfied by polynomials in this sequence. As
a novel contribution to the literature, we provide certain three term
recurrence formula with rational coefficients satisfied by $\mathbb{U}%
_{n}^{(a)}(x;q,j)$, which paves the way to establish an appealing
generalization of the so-called $J$-fractions to the framework of
Sobolev-type orthogonality.

\medskip


\medskip


\medskip
\end{abstract}

%%%%%%%%%%%%%%%%%%%%%%%%%%%%%%%%%%%%%%%%%%

\section{Introduction}

\label{S1-Intro}

%%%%%%%%%%%%%%%%%%%%%%%%%%%%%%%%%%%%%%%%%%%%%%%%%%%%%%%%%%%%%%%%%%%%%%%%%%%%%%%%%%%%%%%%%%%%%%%%%%%%%%%

%%%%%%%%%%%%%%%%%%%%%%%%%%%%%%%%%%%%%%%%%%%%%%%%%%%%%%%%%%%%%%%%%%%%%%%%%%%%%%%%%%%%%%%%%%%%%%%%%%%%%%%

The Al-Salam--Carlitz I and II orthogonal polynomials of degree $n$, usually
denoted in the literature as $U_{n}^{(a)}(x;q)$ and $V_{n}^{(a)}(x;q)$
respectively, are two systems of one parameter $q$-hypergeometric
polynomials introduced in 1965 by W. A. Al-Salam and L. Carlitz, in their
seminal work \cite{AsC-MN65}. Here, $q\in\mathbb{R}$ stands for a fixed
parameter, being the polynomials written in the variable $x$. There is a
straightforward relationship between $U_{n}^{(a)}(x;q) $ and $%
V_{n}^{(a)}(x;q)$ (see \cite[Ch. VI, §10, pp. 195--198]{Chi-78}) 
\begin{equation*}
U_{n}^{(a)}(x;q^{-1})=V_{n}^{(a)}(x;q),
\end{equation*}%
and they are known to be positive definite orthogonal polynomial sequences
for $a<0$, and $a>0$ respectively. Here, and throughout the paper, we assume
the parameter $q$ is such that $0<q<1$, which implies that these two
families belong to the class of orthogonal polynomial solutions of certain
second order $q$-difference equations, known in the literature as the 
\textit{Hahn class} (see \cite{H-MN49}, \cite{koek2010}). In fact, as we
show later on, they can be explicitly given in terms of basic hypergeometric
series. Given the close relation between these two families, and for the
sake of clarity, in this paper we will focus on the Al-Salam--Carlitz I
orthogonal polynomials $\{U_{n}^{(a)}(x;q)\}_{n\geq 0}$. The results
obtained can be also stated for the Al-Salam--Carlitz II orthogonal
polynomials by replacing the parameter $q$ by $q^{-1}$, so we omit explicit
details concerning this second family of orthogonal polynomials.

\smallskip

The Al-Salam--Carlitz I polynomials are orthogonal on the interval $[a,1]$,
with a quite simple $q$-lattice, which makes them suitable for the study to
be carried out hereafter, and also they are of interest in their own right.
For example, they are known to be proportional to the eigenfunctions of
certain quantum mechanical $q$-harmonic oscillators. In \cite{AS-LMP93}, it
is clearly shown that many properties of this $q$-oscillators can be
obtained from the properties of the Al-Salam--Carlitz I orthogonal
polynomials. They are also known to be birth and death process polynomials (%
\cite[Sec. 18.2]{Ismail-05}), with birth rate $aq^{n}$ and death rate $%
1-q^{n}$, and for $a=0$ they become the well known Rogers-Szeg\H{o}
polynomials, of deep implications in the study of the celebrated
Askey-Wilson integral (see, for example \cite{AI-SPM83}, \cite{IS-CJM88}).

\smallskip

On the other hand, in the last decades, the so called \textit{Sobolev
orthogonal polynomials} have attracted the attention of many researchers.
Firstly, this name was given to those families of polynomials orthogonal
with respect to inner products involving positive Borel measures supported
on infinite subsets of the real line, and also involving regular
derivatives. When these derivatives appear only on function evaluations on a
finite discrete set, the corresponding families are called \textit{%
Sobolev-type} or \textit{discrete Sobolev} orthogonal polynomial sequences.
For a recent and comprehensive survey on the subject, see \cite{MX-EM15} and
the references therein. In the last decade of the past century, H. Bavinck
introduced the study of inner products involving differences (instead of
regular derivatives) in uniform lattices on the real line (see \cite%
{B-JCAM95}, \cite{B-AA95}, \cite{B-IM96}, and also \cite{HS-NA19}\ for
recent results on this topic). By analogy with the continuous case, these
are also called Sobolev-type or discrete Sobolev inner products. In
contrast, they are defined on uniform lattices. As a generalization of this
last matter, here we focus on a particular Sobolev-type inner product
defined on a $q$-lattice, instead of on a uniform lattice. This has already
been considered in other works (see, for example in \cite{CS-JDE18} for only
one $q$-derivative). In the present study, we consider an arbitrary number $%
j\in\mathbb{N}$, $j\geq 1$ of $q$-derivatives in the discrete part of the
inner product. For an interesting related work to this paper, see for
example the preprint \cite{FMM-arXiv20}, which appeared just a few days ago
while we were giving the finishing touches to the present manuscript. There,
the authors generalize the action of an arbitrary number of $q$-derivatives
for general orthogonality measures, using the same techniques as for example
in \cite{GGH-M20}, and also in the present paper. It is also worth
mentioning the nice variation considering special non-uniform lattices
(snul), instead of uniform or $q$-lattices, studied in the recent work \cite%
{R-AA20}.

\smallskip

Having said all that, and to the best of our knowledge, an arbitrary number
of $q$-derivatives acting at the same time on the two boundaries of a
bounded orthogonality interval, has never been previously considered in the
literature, and the present work is intended to be a first step in this
direction. This reveals some small differences of the corresponding
polynomial sequences, for example related with the parity of the
polynomials, with respect to what happens considering only one mass point
(as in \cite{FMM-arXiv20}), and that we have right now under study. To be
more precise, this paper deals with the sequence of monic $q$-polynomials $\{%
\mathbb{U}_{n}^{(a)}(x;q,j)\}_{n\geq 0}$, orthogonal with respect to the
Sobolev-type inner product%
\begin{eqnarray}
\left\langle f,g\right\rangle _{\lambda ,\mu }
&=&\int_{a}^{1}f(x;q)g(x;q)(qx,a^{-1}qx;q)_{\infty }d_{q}x
\label{q-SobtypInnPr} \\
&&\quad +\lambda ({\mathscr D}_{q}^{j}f)\left( a;q\right) ({\mathscr D}%
_{q}^{j}g)\left( a;q\right) +\mu ({\mathscr D}_{q}^{j}f)\left( 1;q\right) ({%
\mathscr D}_{q}^{j}g)\left( 1;q\right) ,  \notag
\end{eqnarray}%
where $(qx,a^{-1}qx;q)_{\infty }d_{q}x$ is the orthogonality measure
associated to the Al-Salam--Carlitz I orthogonal polynomials, $a<0$, $%
\lambda ,\mu \in \mathbb{R}_{+}$ and $({\mathscr D}_{q}f)$ denotes the $q$%
-derivative operator, as defined below in (\ref{q-derivative}). It is worth
noting that the above inner product involves an arbitrary number of $q$%
-derivatives on function evaluations on the discrete points $x=a$ and $x=1$,
exclusively. We observe such points conform the boundary of the
orthogonality interval of the Al-Salam--Carlitz I orthogonal polynomials.
Thus, as an extension of the language used in literature, throughout this
manuscript we will refer to $\mathbb{U}_{n}^{(a)}(x;q,j)$ as \textit{%
Al-Salam--Carlitz I-Sobolev type orthogonal polynomials of higher order},
and for the sake of brevity, in what follows we just write $\mathbb{U}%
_{n}^{(a)}(x;q,j)=\mathbb{U}_{n}^{(a)}(x;q)$. We provide here two explicit
representations for $\mathbb{U}_{n}^{(a)}(x;q)$, one as a linear combination
of two consecutive Al-Salam--Carlitz I orthogonal polynomials $%
U_{n}^{(a)}(x;q)$ and $U_{n-1}^{(a)}(x;q)$, and the other one as a $%
_{3}\phi_{2}$ $q$-hypergeometric series, which was unknown so far. This
basic hypergeometric character is always $_{3}\phi _{2}$, with independence
of the number $j$ of $q$-derivatives considered in (\ref{q-SobtypInnPr}).
Next, we obtain two different versions of the structure relation satisfied
by the Sobolev-type $q$-orthogonal polynomials in $\mathbb{U}_{n}^{(a)}(x;q)$%
, and next we use them to obtain closed expressions for the corresponding 
\textit{ladder (creation and annihilation) }$q$-\textit{difference operators}%
. As an application of these ladder $q$-difference operators, we obtain a
three-term recurrence formula with rational coefficients, which allows us to
find every polynomial $\mathbb{U}_{n+1}^{(a)}(x;q)$ of precise degree $n+1$,
in terms of the previous two consecutive polynomials of the same sequence $%
\mathbb{U}_{n}^{(a)}(x;q)$ and $\mathbb{U}_{n-1}^{(a)}(x;q)$, and up to four
different versions of the linear second order $q$-difference equation
satisfied by $\mathbb{U}_{n}^{(a)}(x;q)$.

\smallskip

In the work, we provide four different versions of the second order $q-$%
difference equations satisfied by the family of orthogonal polynomials under
consideration. Also, two different representations of such polynomials are
determined: one as a linear combination of standard Al-Salam--Carlitz I
orthogonal polynomials, and a second one as $_{3}\phi _{2}$ series. Two
versions of structure relations are obtained, in contrast to standard
results, giving rise to four second order $q-$difference equations satisfied
by the elements of this family. The previous results, which clarify the
enriching structure of such polynomials, are followed by novel results using
a non-standard technique to achieve a three-term recurrence formula, and
leading to the appearance of the polynomials under consideration as the
numerators and denominators of the convergents of certain J-fractions.

\smallskip

The manuscript is organized as follows. In Section \ref{S2-DefNot}, we
recall some basic definitions and notations of the $q$-calculus theory, as
well as the basic properties of the Al-Salam--Carlitz I polynomials. In
Section \ref{S3-ConnForm}, we obtain some connection formulas and the basic
hypergeometric representation of the Al-Salam--Carlitz I-Sobolev type
orthogonal polynomials of higher order. Section \ref{S4-LOp} is focused on
two structure relations for the sequence $\{\mathbb{U}_{n}^{(a)}(x;q)\}_{n%
\geq 0}$ , as well as the two different versions of the aforementioned three
term recurrence formula with rational coefficients that $\mathbb{U}%
_{n}^{(a)}(x;q)$ satisfies. In Section \ref{S5-LdE}, combining the
connection formula for $\mathbb{U}_{n}^{(a)}(x;q)$, and the structure
relations obtained in the preceeding Sections, we provide the $q$-difference
ladder operators and four versions of the second linear $q$-difference
equation that the Al-Salam--Carlitz I-Sobolev type polynomials of higher
order satisfy. The work ends with two brief sections on further results. The
first one describes results relating Al-Salam--Carlitz I-Sobolev type
polynomials with Jacobi fractions, and the second illustrates the form of
such polynomials together with some important remarks. A final section on
conclusions and future research problems is also included.

%Finally, the conclusions some open problems are summarized in Section \ref{S6-ConclOP}.

%%%%%%%%%%%%%%%%%%%%%%%%%%%%%%%%%%%%%%%%%%%%%%%%%%%%%%%%%%%%%%%%%%%%%%%%%%%%%%%%%%%%%%%%%%%%%%%%%%%%%%%

%%%%%%%%%%%%%%%%%%%%%%%%%%%%%%%%%%%%%%%%%%%%%%%%%%%%%%%%%%%%%%%%%%%%%%%%%%%%%%%%%%%%%%%%%%%%%%%%%%%%%%%

\section{Definitions and notations}

\label{S2-DefNot}

%%%%%%%%%%%%%%%%%%%%%%%%%%%%%%%%%%%%%%%%%%%%%%%%%%%%%%%%%%%%%%%%%%%%%%%%%%%%%%%%%%%%%%%%%%%%%%%%%%%%%%%

%%%%%%%%%%%%%%%%%%%%%%%%%%%%%%%%%%%%%%%%%%%%%%%%%%%%%%%%%%%%%%%%%%%%%%%%%%%%%%%%%%%%%%%%%%%%%%%%%%%%%%%

This first part of the section is twofold. A first subsection provides the
main tools used in the framework of $q$-calculus, in order to make our
exposition be self-contained. Afterwards, we describe known facts on
Al-Salam--Carlitz I polynomials.

\subsection{$q$-calculus review}

%%%%%%%%%%%%%%%%%%%%%%%%%%%%%%%%%%%%%%%%%%%%%%%%%%%%%%%%%%%%%%%%%%%%%%%%%%%%%%%%%%%%%%%%%%%%%%%%%%%%%%%

%%%%%%%%%%%%%%%%%%%%%%%%%%%%%%%%%%%%%%%%%%%%%%%%%%%%%%%%%%%%%%%%%%%%%%%%%%%%%%%%%%%%%%%%%%%%%%%%%%%%%%%

For every $q\neq 0$ and $q\neq 1$, the $q$-number, $q$-bracket, or simply
the \textit{basic number} $[n]_{q}$, is defined by \cite{koek2010}, \cite%
{NikUvSu-91}%
\begin{equation*}
[0]_{q}=0,\quad [n]_{q}=\frac{1-q^{n}}{1-q}=\sum_{k=0}^{n-1}q^{k},\quad
n=1,2,3,\ldots ,
\end{equation*}%
which comes from the equality%
\begin{equation*}
\lim_{q\rightarrow 1}\frac{1-q^{n}}{1-q}=n.
\end{equation*}%
In this framework, a $q$-analogue of the factorial of $n$ is given by%
\begin{equation*}
[0]_{q}!=1,\quad [n]_{q}!=[n]_{q}[n-1]_{q}\cdots [2]_{q}[1]_{q},\quad
n=1,2,3,\ldots ,
\end{equation*}%
and we can also give a $q$-analogue of the well known Pochhammer symbol, or
shifted factorial (see \cite{koek2010}). For $n=1,2,3,\ldots $, we have%
\begin{equation*}
\left( a;q\right) _{0}=1,\quad \left( a;q\right) _{n}=(1-a)(1-aq)\cdots
(1-aq^{n-1})=\prod\limits_{i=1}^{n}(1-aq^{i-1}).
\end{equation*}%
Moreover, we use the following notation%
\begin{equation*}
(a_{1},\ldots ,a_{r};q)_{n}=\prod_{k=1}^{r}(a_{k};q)_{n}.
\end{equation*}

The following definitions, also in the framework of the $q$-calculus, can be
found in \cite{koek2010}. The basic hypergeometric, or $q$-hypergeometric
series $_{r}\phi _{s}$, is defined as follows. Let $\left\{ a_{i}\right\}
_{i=1}^{r}$ and $\left\{ b_{i}\right\} _{i=1}^{s}$ be complex numbers such
that $b_{i}\neq q^{-n}$ for $n\in \mathbb{N}$. We write 
\begin{equation*}
_{r}\phi _{s}\left( 
\begin{array}{c}
a_{1},a_{2},\ldots ,a_{r} \\ 
b_{1},b_{2},\ldots ,b_{s}%
\end{array}%
;q,z\right) =\sum_{k=0}^{\infty }\frac{\left( a_{1},\ldots ,a_{r};q\right)
_{k}}{\left( b_{1},\ldots ,b_{s};q\right) _{k}}\left( (-1)^{k}q^{\binom{k}{2}%
}\right) ^{1+s-r}\frac{z^{k}}{\left( q;q\right) _{k}}.
\end{equation*}

The $q$-binomial coefficient is given by%
\begin{equation*}
\begin{bmatrix}
k \\ 
n%
\end{bmatrix}%
_{q}=\frac{(q;q)_{n}}{(q;q)_{k}(q;q)_{n-k}}=\frac{[n]_{q}!}{%
[k]_{q}![n-k]_{q}!}=%
\begin{bmatrix}
k \\ 
n-k%
\end{bmatrix}%
_{q},\quad k=0,1,\ldots ,n,
\end{equation*}%
where $n$ denotes a nonnegative integer.

Concerning the $q$-analog of the derivative operator, we have the $q$%
-derivative, or the Euler--Jackson $q$-difference operator%
\begin{equation}
({\mathscr D}_{q}f)(z)=%
\begin{cases}
\displaystyle\frac{f(qz)-f(z)}{(q-1)z}, & \text{if}\ z\neq 0,\ q\neq 1, \\ 
&  \\ 
f^{\prime }(z), & \text{otherwise},%
\end{cases}
\label{q-derivative}
\end{equation}%
where ${\mathscr D}_{q}^{0}f=f$, ${\mathscr D}_{q}^{n}f={\mathscr D}_{q}({%
\mathscr D}_{q}^{n-1}f)$, for $n\geq 1$, and 
\begin{equation*}
\lim\limits_{q\rightarrow 1}{\mathscr D}_{q}f(z)=f^{\prime }(z).
\end{equation*}%
Moreover, one has the following properties%
\begin{equation}
{\mathscr D}_{q}[f(\gamma z)]=\gamma ({\mathscr D}_{q}f)(\gamma z),\quad
\forall \,\,\gamma \in \mathbb{C},  \label{cadrule}
\end{equation}%
\begin{equation}
{\mathscr D}_{q}f(z)={\mathscr D}_{q^{-1}}f(qz)\Leftrightarrow {\mathscr D}%
_{q^{-1}}f(z)={\mathscr D}_{q}f(q^{-1}z),  \label{DqProp}
\end{equation}%
\begin{eqnarray}
{\mathscr D}_{q}[f(z)g(z)] &=&f(qz){\mathscr D}_{q}g(z)+g(z){\mathscr D}%
_{q}f(z)  \label{prodqD} \\
&=&f(z){\mathscr D}_{q}g(z)+g(qz){\mathscr D}_{q}f(z),  \notag
\end{eqnarray}%
and the following interesting property which can be found in \cite[p. 104]%
{KS-JCAM01}%
\begin{equation}
{\mathscr D}_{q^{-1}}({\mathscr D}_{q}f)(z)=q{\mathscr D}_{q}({\mathscr D}%
_{q^{-1}}f)(z),  \label{DqProOK}
\end{equation}%
This $q$-derivative operator leads to define a $q$-analogue of Leibniz' rule%
\begin{equation}
{\mathscr D}_{q}^{n}[f(z)g(z)]=\sum_{k=0}^{n}%
\begin{bmatrix}
k \\ 
n%
\end{bmatrix}%
_{q}({\mathscr D}_{q}^{k}f)(z)\cdot({\mathscr D}_{q}^{n-k}g)(q^{k}z),\quad
n=0,1,2,\ldots .  \label{LRule}
\end{equation}

Of special interest is the way in which the integral form of the inner
product (\ref{q-SobtypInnPr}) is defined, corresponding to the so called
Jackson $q$-integral, given by%
\begin{equation*}
\int_{0}^{z}f(x)d_{q}x=(1-q)z\sum_{k=0}^{\infty }q^{k}f(q^{k}z),
\end{equation*}%
which in a generic interval $[a,b]$ is given by%
\begin{equation*}
\int_{a}^{b}f(x)d_{q}x=\int_{0}^{b}f(x)d_{q}x-\int_{0}^{a}f(x)d_{q}x.
\end{equation*}

The following definition will also be needed throughout the paper. The
Jackson--Hahn--Cigler $q$-subtraction is given by (see, for example \cite[%
Def. 6]{E-PEAS09}, and the references given there)%
\begin{equation*}
\left( x\boxminus _{q}y\right) ^{n}=\prod_{j=0}^{n-1}\left( x-yq^{j}\right)
=x^{n}(y/x;q)_{n}=\sum_{k=0}^{n}%
\begin{bmatrix}
k \\ 
n%
\end{bmatrix}%
_{q}q^{\binom{k}{2}}(-y)^{k}x^{n-k}.
\end{equation*}

Finally, we recall here the $q$-Taylor formula (see \cite[Th. 6.3]%
{srivastava2011zeta}), with the Cauchy remainder term, which is defined by%
\begin{equation*}
f(x)=\sum_{k=0}^{n}\frac{({\mathscr D}_{q}^{k}f)(a)}{[k]_{q}!}(x\boxminus
_{q}a)^{k}+\frac{1}{[n]_{q}!}\int_{a}^{x}({\mathscr D}_{q}^{n+1}f)(t)\cdot(x%
\boxminus _{q}qt)^{n}d_{q}t.
\end{equation*}

The promising idea of $q-$derivatives is also potentially applicable in
rings, see~\cite{shang1,shang2}.

%%%%%%%%%%%%%%%%%%%%%%%%%%%%%%%%%%%%%%%%%%%%%%%%%%%%%%%%%%%%%%%%%%%%%%%%%%%%%%%%%%%%%%%%%%%%%%%%%%%%%%%

%%%%%%%%%%%%%%%%%%%%%%%%%%%%%%%%%%%%%%%%%%%%%%%%%%%%%%%%%%%%%%%%%%%%%%%%%%%%%%%%%%%%%%%%%%%%%%%%%%%%%%%

\subsection{Al-Salam--Carlitz I orthogonal polynomials}

%%%%%%%%%%%%%%%%%%%%%%%%%%%%%%%%%%%%%%%%%%%%%%%%%%%%%%%%%%%%%%%%%%%%%%%%%%%%%%%%%%%%%%%%%%%%%%%%%%%%%%%

%%%%%%%%%%%%%%%%%%%%%%%%%%%%%%%%%%%%%%%%%%%%%%%%%%%%%%%%%%%%%%%%%%%%%%%%%%%%%%%%%%%%%%%%%%%%%%%%%%%%%%%

After the above $q$-calculus introduction, we continue by giving several
aspects and properties of the Al-Salam--Carlitz I polynomials $%
\{U_{n}^{(a)}(x;q)\}_{n\geq 0}$ . All the elements presented in the
remaining of this section can be found in \cite[Ch. VI, §10]{Chi-78}, \cite[%
Sec. 18.2]{Ismail-05}, \cite[Sec. 14.24]{koek2010}, and \cite{GR-Basic04},
among other references. Such polynomials are orthogonal with respect to the
inner product on $\mathbb{P}$, the linear space of polynomials with real
coefficients%
\begin{equation*}
\left\langle f,g\right\rangle =\int_{a}^{1}f(x;q)g(x;q)d\alpha ^{(a)},\quad
a<0,
\end{equation*}%
where $d\alpha ^{(a)}=(qx,a^{-1}qx;q)_{\infty }\,d_{q}x$, which implies that 
$\alpha ^{(a)}$ is a step function on $[a,1]$ with jumps%
\begin{equation*}
\frac{q^{k}}{(aq;q)_{\infty }(q,q/a;q)_{k}}\quad \text{at the points }x=q^{k}%
\text{, }k=0,1,2,\ldots ,
\end{equation*}%
and jumps%
\begin{equation*}
\frac{-aq^{k}}{(q/a;q)_{\infty }(q,aq;q)_{k}}\quad \text{at the points }%
x=aq^{k}\text{, }k=0,1,2,\ldots .
\end{equation*}%
It can be easily checked that, when $a=-1$ the above inner product becomes
the inner product associated to the discrete $q$-Hermite orthogonal
polynomials.

The Al-Salam--Carlitz I polynomials can be explicitly given by%
\begin{equation}
U_{n}^{(a)}(x;q)=(-a)^{n}q^{\binom{n}{2}}\,_{2}\phi _{1}\left( 
\begin{array}{c}
q^{-n},x^{-1} \\ 
0%
\end{array}%
;q,a^{-1}qx\right) ,\quad a<0,  \label{ASP}
\end{equation}%
satisfying the orthogonality relation%
\begin{equation*}
\int_{a}^{1}U_{m}^{(a)}(x;q)U_{n}^{(a)}(x;q)d\alpha
^{(a)}=(1-a)(-a)^{n}(q;q)_{n}q^{\binom{n}{2}}\delta _{m,n}
\end{equation*}

%%%%%%%%%%%%%%%%%%%%%%%%%%%%%%%%%%%%%%%%%%%%%%%%%%%%%%%%%%%%%%%%%%%%%%%%%%%%%%%%%%%%%%%%%%%%%%%%%%%%%%%

%%%%%%%%%%%%%%%%%%%%%%%%%%%%%%%%%%%%%%%%%%%%%%%%%%%%%%%%%%%%%%%%%%%%%%%%%%%%%%%%%%%%%%%%%%%%%%%%%%%%%%%

\begin{proposition}
\label{S1-Proposition11} Let $\{U_{n}^{(a)}(x;q)\}_{n\geq 0}$ be the
sequence of Al-Salam--Carlitz I polynomials of degree $n$. The following
statements hold:

\begin{enumerate}
\item The recurrence relation \cite[Section 14.24]{koek2010} 
\begin{equation}
xU_{n}^{(a) }(x;q) =U_{n+1}^{\left( a\right) }\left( x;q\right)
+\beta_nU_{n}^{\left( a\right) }\left( x;q\right) +\gamma_n U_{n-1}^{\left(
a\right) }\left( x;q\right),  \label{ReR}
\end{equation}
with initial conditions $U_{-1}^{(a)}(x;q)=0$ and $U_{0}^{(a)}(x;q)=1$.
Here, $\beta_n=\left( a+1\right) q^{n}$ and $\gamma_n=-aq^{n-1}\left(
1-q^{n}\right)$.

\item Structure relation \cite{KS-JCAM01}. For every $n\in \mathbb{N}$,%
\begin{equation}
\sigma (x){\mathscr D}_{q^{-1}}U_{n}^{(a)}\left( x;q\right) =\overline{%
\alpha }_{n}U_{n+1}^{\left( a\right) }\left( x;q\right) +\overline{\beta }%
_{n}U_{n}^{\left( a\right) }\left( x;q\right) +\overline{\gamma }%
_{n}U_{n-1}^{\left( a\right) }\left( x;q\right) ,  \label{StruR}
\end{equation}%
where $\sigma (x)=(x-1)(x-a)$, $\overline{\alpha }_{n}=q^{1-n}\left[ n\right]
_{q}$, $\overline{\beta }_{n}=\left( a+1\right) q\left[ n\right] _{q}$ and $%
\overline{\gamma }_{n}=aq^{n}\left[ n\right] _{q}$.

\item Squared norm \cite[Section 14.24]{koek2010}. For every $n\in \mathbb{N}
$,%
\begin{equation*}
||U_{n}^{\left( a\right) }||^{2}=\left( -a\right) ^{n}\left( 1-q\right)
\left( q;q\right) _{n}\left( q,a,a^{-1}q;q\right) _{\infty }q^{\binom{n}{2}}.
\end{equation*}

\item Forward shift operator \cite[Section 14.24]{koek2010}%
\begin{equation}
{\mathscr D}_{q}^{k}U_{n}^{\left( a\right) }\left( x;q\right) =\left[ n%
\right] _{q}^{\left( k\right) }U_{n-k}^{\left( a\right) }\left( x;q\right) ,
\label{FSop}
\end{equation}%
where 
\begin{equation*}
\left[ n\right] _{q}^{(k)}=\frac{(q^{-n};q)_{k}}{(q-1)^{k}}q^{kn-\binom{k}{2}%
},
\end{equation*}%
denote the $q$-falling factorial \cite{AS-JEDA13}.

\item Second-order $q$-difference equation \cite{DA-JPA05} 
\begin{equation*}
\sigma(x){\mathscr D}_q{\mathscr D}_{q^{-1}}U_{n}^{(a) }(x;q)+\tau(x){%
\mathscr D}_{q}U_{n}^{(a) }(x;q)+\lambda_{n,q}U_{n}^{(a) }(x;q)=0,
\end{equation*}
where $\tau(x)=(x-a-1)/(1-q)$ and $\lambda_{n,q}=[n]_q([1-n]_q\sigma^{\prime
\prime }/2-\tau^{\prime })$.
\end{enumerate}
\end{proposition}

%%%%%%%%%%%%%%%%%%%%%%%%%%%%%%%%%%%%%%%%%%%%%%%%%%%%%%%%%%%%%%%%%%%%%%%%%%%%%%%%%%%%%%%%%%%%%%%%%%%%%%%

%%%%%%%%%%%%%%%%%%%%%%%%%%%%%%%%%%%%%%%%%%%%%%%%%%%%%%%%%%%%%%%%%%%%%%%%%%%%%%%%%%%%%%%%%%%%%%%%%%%%%%%

\begin{proposition}[Christoffel-Darboux formula]
\label{S1-Proposition12} Let $\{U_{n}^{(a)}(x;q)\}_{n\geq 0}$ be the
sequence of Al-Salam--Carlitz I polynomials. If we denote the $n$-th
reproducing kernel by 
\begin{equation*}
K_{n,q}(x,y)=\sum_{k=0}^{n}\frac{U_{k}^{(a)}(x;q)\cdot U_{k}^{(a)}(y;q)}{%
||U_{k}^{(a)}||^{2}}.
\end{equation*}%
Then, for all $n\in \mathbb{N}$, it holds that 
\begin{equation}
K_{n,q}(x,y)=\frac{U_{n+1}^{\left( a\right) }\left( x;q\right)\cdot
U_{n}^{\left( a\right) }\left( y;q\right) -U_{n+1}^{\left( a\right) }\left(
y;q\right)\cdot U_{n}^{\left( a\right) }\left( x;q\right) }{\left(
x-y\right) ||U_{n}^{\left( a\right) }||^{2}}.  \label{CDarb}
\end{equation}
\end{proposition}

%%%%%%%%%%%%%%%%%%%%%%%%%%%%%%%%%%%%%%%%%%%%%%%%%%%%%%%%%%%%%%%%%%%%%%%%%%%%%%%%%%%%%%%%%%%%%%%%%%%%%%%

%%%%%%%%%%%%%%%%%%%%%%%%%%%%%%%%%%%%%%%%%%%%%%%%%%%%%%%%%%%%%%%%%%%%%%%%%%%%%%%%%%%%%%%%%%%%%%%%%%%%%%%

Concerning the partial $q$-derivatives of $K_{n,q}(x,y)$, we use the
following notation 
\begin{eqnarray*}
K_{n,q}^{(i,j)}(x,y) &=&{\mathscr D}_{q,y}^{i}({\mathscr D}%
_{q,x}^{j}K_{n,q}(x,y)) \\
&=&\sum_{k=0}^{n}\frac{{\mathscr D}_{q}^{i}U_{k}^{(a)}(x;q)\cdot{\mathscr D}%
_{q}^{j}U_{k}^{(a)}(y;q)}{||U_{k}^{(a)}||^{2}}.
\end{eqnarray*}

%%%%%%%%%%%%%%%%%%%%%%%%%%%%%%%%%%%%%%%%%%%%%%%%%%%%%%%%%%%%%%%%%%%%%%%%%%%%%%%%%%%%%%%%%%%%%%%%%%%%%%%

Next, we provide a technical result that will be useful later on.

%%%%%%%%%%%%%%%%%%%%%%%%%%%%%%%%%%%%%%%%%%%%%%%%%%%%%%%%%%%%%%%%%%%%%%%%%%%%%%%%%%%%%%%%%%%%%%%%%%%%%%%

\begin{lemma}
\label{S1-LemmaKernel0j} Let $\{U_{n}^{(a)}(x;q)\}_{n\geq 0}$ be the
sequence of Al-Salam--Carlitz I polynomials of degree $n$. Then following
statements hold, for all $n\in \mathbb{N}$, 
\begin{equation}
K_{n-1,q}^{(0,j)}(x,y)={\mathcal{A}}_{n}(x,y)U_{n}^{(a)}(x;q)+{\mathcal{B}}%
_{n}(x,y)U_{n-1}^{(a)}(x;q),  \label{Kernel0j}
\end{equation}%
where 
\begin{equation*}
{\mathcal{A}}_{n}(x,y)=\frac{\left[ j\right] _{q}!}{||U_{n-1}^{(a)}||^{2}%
\left( x\boxminus _{q}y\right) ^{j+1}}\sum_{k=0}^{j}\frac{{\mathscr D}%
_{q}^{k}U_{n-1}^{(a)}(y;q)}{\left[ k\right] _{q}!}(x\boxminus _{q}y)^{k},
\end{equation*}%
and 
\begin{equation*}
{\mathcal{B}}_{n}(x,y)=-\frac{\left[ j\right] _{q}!}{||U_{n-1}^{(a)}||^{2}%
\left( x\boxminus _{q}y\right) ^{j+1}}\sum_{k=0}^{j}\frac{{\mathscr D}%
_{q}^{k}U_{n}^{(a)}(y;q)}{\left[ k\right] _{q}!}(x\boxminus _{q}y)^{k}.
\end{equation*}
\end{lemma}

%%%%%%%%%%%%%%%%%%%%%%%%%%%%%%%%%%%%%%%%%%%%%%%%%%%%%%%%%%%%%%%%%%%%%%%%%%%%%%%%%%%%%%%%%%%%%%%%%%%%%%%

%%%%%%%%%%%%%%%%%%%%%%%%%%%%%%%%%%%%%%%%%%%%%%%%%%%%%%%%%%%%%%%%%%%%%%%%%%%%%%%%%%%%%%%%%%%%%%%%%%%%%%%

\begin{proof}
Applying the $j$-th $q$-derivative to (\ref{CDarb}) with respect to $y$
yields%
\begin{equation*}
K_{q,n-1}^{\left( 0,j\right) }\left( x,y\right) =\frac{1}{%
||U_{n-1}^{(a)}||^{2}}\times
\end{equation*}%
\begin{equation}
\left( U_{n}^{(a)}(x;q){\mathscr D}_{q,y}^{j}\left( \frac{U_{n-1}^{(a)}(y;q)%
}{x-y}\right) -U_{n-1}^{(a)}(x;q){\mathscr D}_{q,y}^{j}\left( \frac{%
U_{n}^{(a)}(y;q)}{x-y}\right) \right) .  \label{K0jI}
\end{equation}%
Then, using the $q$-analogue of Leibniz' rule (\ref{LRule}) and%
\begin{equation*}
{\mathscr D}_{q,y}^{n}\left( \frac{1}{x-y}\right) =\frac{\left[ n\right]
_{q}!}{\left( x\boxminus _{q}y\right) ^{n+1}},
\end{equation*}%
we deduce%
\begin{eqnarray*}
{\mathscr D}_{q,y}^{j}\left( \frac{U_{n-1}^{(a)}(y;q)}{x-y}\right)
&=&\sum_{k=0}^{j}%
\begin{bmatrix}
j \\ 
k%
\end{bmatrix}%
_{q}{\mathscr D}_{q}^{k}U_{n-1}^{(a)}(y;q)\cdot{\mathscr D}%
_{q,y}^{j-k}\left( \frac{1}{x-q^{k}y}\right) \\
&=&\sum_{k=0}^{j}\frac{\left[ j\right] _{q}!}{\left[ k\right] _{q}!}\frac{{%
\mathscr D}_{q}^{k}U_{n-1}^{(a)}(y;q)}{\left( x\boxminus _{q}q^{k}y\right)
^{j-k+1}}.
\end{eqnarray*}%
Next, it is easy to check that%
\begin{equation*}
(x\boxminus _{q}q^{k}y)^{j-k+1}=\frac{(x\boxminus _{q}y)^{j+1}}{(x\boxminus
_{q}y)^{k}},
\end{equation*}%
and therefore we have%
\begin{equation*}
{\mathscr D}_{q,y}^{j}\left( \frac{U_{n-1}^{(a)}(y;q)}{x-y}\right) =\frac{%
\left[ j\right] _{q}!}{\left( x\boxminus _{q}y\right) ^{j+1}}\sum_{k=0}^{j}%
\frac{{\mathscr D}_{q}^{k}U_{n-1}^{(a)}(y;q)}{\left[ k\right] _{q}!}\left(
x\boxminus _{q}y\right) ^{k}.
\end{equation*}%
Finally, combining all the above with (\ref{K0jI}) we obtain (\ref{Kernel0j}%
). This completes the proof.
\end{proof}

\begin{remark}
We observe that%
\begin{equation*}
K_{q,n-1}^{\left( 0,j\right) }\left( x,\alpha \right) =\frac{\left[ j\right]
_{q}!}{||U_{n-1}^{(a)}||^{2}\left( x\boxminus _{q}y\right) ^{j+1}}\left(
U_{n}^{(a)}(x;q)\cdot Q_{q,j}(x,\alpha
,U_{n-1}^{(a)})-U_{n-1}^{(a)}(x;q)\cdot Q_{q,j}(x,\alpha
,U_{n}^{(a)})\right) ,
\end{equation*}%
where $Q_{q,j}(x,\alpha ,U_{n-1}^{(a)})$ and $Q_{q,j}(x,\alpha ,U_{n}^{(a)})$
denote the $q$-Taylor polynomials of degree $j$ of the polynomials $%
U_{n-1}^{(a)}(x;q)$ y $U_{n}^{(a)}(x;q)$ at $x=\alpha $, respectively.
\end{remark}

%%%%%%%%%%%%%%%%%%%%%%%%%%%%%%%%%%%%%%%%%%%%%%%%%%%%%%%%%%%%%%%%%%%%%%%%%%%%%%%%%%%%%%%%%%%%%%%%%%%%%%%

%%%%%%%%%%%%%%%%%%%%%%%%%%%%%%%%%%%%%%%%%%%%%%%%%%%%%%%%%%%%%%%%%%%%%%%%%%%%%%%%%%%%%%%%%%%%%%%%%%%%%%%

\section{Connection formulas and hypergeometric representation}

\label{S3-ConnForm}

%%%%%%%%%%%%%%%%%%%%%%%%%%%%%%%%%%%%%%%%%%%%%%%%%%%%%%%%%%%%%%%%%%%%%%%%%%%%%%%%%%%%%%%%%%%%%%%%%%%%%%%

%%%%%%%%%%%%%%%%%%%%%%%%%%%%%%%%%%%%%%%%%%%%%%%%%%%%%%%%%%%%%%%%%%%%%%%%%%%%%%%%%%%%%%%%%%%%%%%%%%%%%%%

In this section we define the Al-Salam--Carlitz I-Sobolev type polynomials
of higher order $\{\mathbb{U}_{n}^{(a)}(x;q)\}_{n\geq 0}$, and describe
different relations which relate them to the Al-Salam--Carlitz I
polynomials. These links will be useful in the sequel. Al-Salam--Carlitz
I-Sobolev type polynomials are defined to be orthogonal with respect to
Sobolev-type inner product 
\begin{eqnarray}
\left\langle f,g\right\rangle _{\lambda ,\mu } &=&{%
\int_{a}^{1}f(x;q)g(x;q)(qx,a^{-1}qx;q)_{\infty }d_{q}x}  \label{piSob} \\
&&\quad +\lambda ({\mathscr D}_{q}^{j}f)\left( a;q\right) \cdot ({\mathscr D}%
_{q}^{j}g)\left( a;q\right) +\mu ({\mathscr D}_{q}^{j}f)\left( 1;q\right)
\cdot ({\mathscr D}_{q}^{j}g)\left( 1;q\right) ,  \notag
\end{eqnarray}
where $a<0$, $\lambda ,\mu \in \mathbb{R}_{+}$, and $j\in \mathbb{N}$, $%
j\geq 1$.

In a first approach, we express $\{\mathbb{U}_{n}^{(a)}(x;q)\}_{n\geq 0}$ in
terms of the Al-Salam--Carlitz I polynomials $\{U_{n}^{(a)}(x;q)\}_{n\geq 0}$%
, the kernel polynomials and their corresponding derivatives. Moreover, we
obtain a representation of the proposed polynomials as hypergeometric
functions. Let us depart from the Fourier expansion 
\begin{equation*}
\mathbb{U}_{n}^{(a)}(x;q)=U_{n}^{(a)}(x;q)+%
\sum_{k=0}^{n-1}a_{n,k}U_{k}^{(a)}(x;q).
\end{equation*}
In view of (\ref{piSob}), and considering the orthogonality properties for $%
U_{n}^{(a)}(x;q)$, for $0\leq k\leq n-1$, the coefficients in the previous
expansion are given by 
\begin{equation*}
a_{n,k}=-\frac{\lambda {\mathscr D}_{q}^{j}\mathbb{U}_{n}^{(a)}(a;q)\cdot{%
\mathscr D}_{q}^{j}U_{n}^{(a)}(a;q)+\mu {\mathscr D}_{q}^{j}\mathbb{U}%
_{n}^{(a)}(1;q)\cdot{\mathscr D}_{q}^{j}U_{n}^{(a)}(1;q)}{||U_{k}^{(a)}||^{2}%
}.
\end{equation*}%
Thus%
\begin{equation*}
\mathbb{U}_{n}^{(a)}(x;q)=U_{n}^{(a)}(x;q)-\lambda {\mathscr D}_{q}^{j}%
\mathbb{U}_{n}^{(a)}(a;q)\cdot K_{n-1,q}^{(0,j)}(x,a)-\mu {\mathscr D}%
_{q}^{j}\mathbb{U}_{n}^{(a)}(1;q)\cdot K_{n-1,q}^{(0,j)}(x,1).
\end{equation*}%
After some manipulations, we obtain a linear system $AX=b$ with two
unknowns, namely ${\mathscr D}_{q}^{j}\mathbb{U}_{n}^{(a)}(a;q)$ and ${%
\mathscr D}_{q}^{j}\mathbb{U}_{n}^{(a)}(1;q)$, where 
\begin{equation*}
A=%
\begin{pmatrix}
1+\lambda K_{n-1,q}^{(j,j)}(a,a) & \mu K_{n-1,q}^{(j,j)}(a,1) \\ 
\lambda K_{n-1,q}^{(j,j)}(1,a) & 1+\mu K_{n-1,q}^{(j,j)}(1,1)%
\end{pmatrix}%
,
\end{equation*}%
and 
\begin{equation*}
X=({\mathscr D}_{q}^{j}\mathbb{U}_{n}^{(a)}(a;q),\,{\mathscr D}_{q}^{j}%
\mathbb{U}_{n}^{(a)}(1;q))^{T},\quad b=({\mathscr D}%
_{q}^{j}U_{n}^{(a)}(a;q),\,{\mathscr D}_{q}^{j}U_{n}^{(a)}(1;q))^{T}.
\end{equation*}%
Cramer's rule yields%
\begin{equation}  \label{e744}
\mathbb{U}_{n}^{(a)}(x;q)=U_{n}^{(a)}(x;q)-\lambda
K_{n-1,q}^{(0,j)}(x,a)\cdot\Delta _{j,n}^{(1)}(a)-\mu
K_{n-1,q}^{(0,j)}(x,1)\cdot\Delta _{j,n}^{(2)}(a),
\end{equation}%
where%
\begin{equation*}
\Delta _{j,n}^{(i)}(a)=%
\begin{cases}
\det (A)^{-1}\det 
\begin{pmatrix}
{\mathscr D}_{q}^{j}U_{n}^{(a)}(a;q) & \mu K_{n-1,q}^{(j,j)}(a,1) \\ 
{\mathscr D}_{q}^{j}U_{n}^{(a)}(1;q) & 1+\mu K_{n-1,q}^{(j,j)}(1,1)%
\end{pmatrix}%
, & \mbox{if }i=1, \\ 
&  \\ 
\det (A)^{-1}\det 
\begin{pmatrix}
1+\lambda K_{n-1,q}^{(j,j)}(a,a) & {\mathscr D}_{q}^{j}U_{n}^{(a)}(a;q) \\ 
\lambda K_{n-1,q}^{(j,j)}(1,a) & {\mathscr D}_{q}^{j}U_{n}^{(a)}(1;q)%
\end{pmatrix}%
, & \mbox{if }i=2,%
\end{cases}%
\end{equation*}%
Hence, we obtain a first connection formula, namely 
\begin{equation}
\mathbb{U}_{n}^{(a)}(x;q)={\mathcal{C}}_{1,n}(x)\cdot U_{n}^{(a)}(x;q)+{%
\mathcal{D}}_{1,n}(x) \cdot U_{n-1}^{(a)}(x;q),  \label{ConexF_I}
\end{equation}%
where 
\begin{equation*}
{\mathcal{C}}_{1,n}(x)=1-\lambda \Delta _{j,n}^{(1)}(a){\mathcal{A}}%
_{n}(x,a)-\mu \Delta _{j,n}^{(2)}(a){\mathcal{A}}_{n}(x,1),
\end{equation*}%
and 
\begin{equation*}
{\mathcal{D}}_{1,n}(x)=-\lambda \Delta _{j,n}^{(1)}(a){\mathcal{B}}%
_{n}(x,a)-\mu \Delta _{j,n}^{(2)}(a){\mathcal{B}}_{n}(x,1).
\end{equation*}
The previous connection formula is obtained after the application of Lemma %
\ref{S1-LemmaKernel0j}, which yields 
\begin{equation*}
K_{n-1,q}^{(0,j)}(x,a)={\mathcal{A}}_{n}(x,a)\cdot U_{n}^{(a)}(x;q)+{%
\mathcal{B}}_{n}(x,a)\cdot U_{n-1}^{(a)}(x;q),
\end{equation*}
together with 
\begin{equation*}
K_{n-1,q}^{(0,j)}(x,1)={\mathcal{A}}_{n}(x,1)\cdot U_{n}^{(a)}(x;q)+{%
\mathcal{B}}_{n}(x,1)\cdot U_{n-1}^{(a)}(x;q).
\end{equation*}
Therefore, from (\ref{e744}) we get%
\begin{eqnarray*}
\mathbb{U}_{n}^{(a)}(x;q) &=&U_{n}^{(a)}(x;q)-\lambda [{\mathcal{A}}%
_{n}(x,a)\cdot U_{n}^{(a)}(x;q)+{\mathcal{B}}_{n}(x,a)\cdot
U_{n-1}^{(a)}(x;q)]\Delta _{j,n}^{(1)}(a) \\
&&\quad -\mu [{\mathcal{A}}_{n}(x,a)\cdot U_{n}^{(a)}(x;q)+{\mathcal{B}}%
_{n}(x,a)\cdot U_{n-1}^{(a)}(x;q)]\Delta _{j,n}^{(2)}(a) \\
&=&[1-\lambda \Delta _{j,n}^{(1)}(a){\mathcal{A}}_{n}(x,a)-\mu \Delta
_{j,n}^{(2)}(a){\mathcal{A}}_{n}(x,1)]U_{n}^{(a)}(x;q) \\
&&\quad +[-\lambda \Delta _{j,n}^{(1)}(a){\mathcal{B}}_{n}(x,a)-\mu \Delta
_{j,n}^{(2)}(a){\mathcal{B}}_{n}(x,1)]U_{n-1}^{(a)}(x;q),
\end{eqnarray*}%
leading to (\ref{ConexF_I}). At this point, we provide another relation
between the two families of polynomials, which will be applied in Theorem~%
\ref{sr_theorem}. More precisely, from (\ref{ConexF_I}) and the recurrence
relation (\ref{ReR}) we have that 
\begin{equation}
\mathbb{U}_{n-1}^{(a)}(x;q)={\mathcal{C}}_{2,n}(x) \cdot U_{n}^{(a)}(x;q)+{%
\mathcal{D}}_{2,n}(x)\cdot U_{n-1}^{(a)}(x;q),  \label{ConexF_II}
\end{equation}%
where%
\begin{equation*}
{\mathcal{C}}_{2,n}(x)=-\frac{{\mathcal{D}}_{1,n-1}(x)}{\gamma _{n-1}},
\end{equation*}%
and%
\begin{equation*}
{\mathcal{D}}_{2,n}(x)={\mathcal{C}}_{1,n-1}(x)+{\mathcal{C}}_{2,n}(x)\left(
\beta _{n-1}-x\right) .
\end{equation*}%
From (\ref{ConexF_I})--(\ref{ConexF_II}) we deduce 
\begin{equation}
U_{n}^{(a)}(x;q)=\det (B_{n}(x))^{-1}\det 
\begin{pmatrix}
\mathbb{U}_{n}^{(a)}(x;q) & \mathbb{U}_{n-1}^{(a)}(x;q) \\ 
{\mathcal{D}}_{1,n}(x) & {\mathcal{D}}_{2,n}(x)%
\end{pmatrix}
\label{ConexF_III}
\end{equation}%
and 
\begin{equation}
U_{n-1}^{(a)}(x;q)=-\det (B_{n}(x))^{-1}\det 
\begin{pmatrix}
\mathbb{U}_{n}^{(a)}(x;q) & \mathbb{U}_{n-1}^{(a)}(x;q) \\ 
{\mathcal{C}}_{1,n}(x) & {\mathcal{C}}_{2,n}(x)%
\end{pmatrix}
\label{ConexF_IV}
\end{equation}%
where 
\begin{equation*}
B_{n}(x)=%
\begin{pmatrix}
{\mathcal{C}}_{1,n}(x) & {\mathcal{C}}_{2,n}(x) \\ 
{\mathcal{D}}_{1,n}(x) & {\mathcal{D}}_{2,n}(x)%
\end{pmatrix}%
.
\end{equation*}

Finally, we focus our attention on the representation of $\mathbb{U}%
_{n}^{(a)}(x;q)$ as hypergeometric functions. A similar analysis to that
carried out in \cite[Th. 2]{CS-JDE18}, yields the following result

%%%%%%%%%%%%%%%%%%%%%%%%%%%%%%%%%%%%%%%%%%%%%%%%%%%%%%%%%%%%%%%%%%%%%%%%%%%%%%%%%%%%%%%%%%%%%%%%%%%%%%%

%%%%%%%%%%%%%%%%%%%%%%%%%%%%%%%%%%%%%%%%%%%%%%%%%%%%%%%%%%%%%%%%%%%%%%%%%%%%%%%%%%%%%%%%%%%%%%%%%%%%%%%

\begin{proposition}[Hypergeometric character]
\label{S3-Theorem31}For $a<0$, the Al-Salam--Carlitz I-Sobolev type
polynomials of higher order $\{\mathbb{U}_{n}^{(a)}(x;q)\}_{n\geq 0}$, have
the following hypergeometric representation:%
\begin{equation}
\mathbb{U}_{n}^{(a)}(x;q)=(-a)^{n}\frac{{\mathcal{D}}_{1,n}(x)(1-\psi
_{n}(x)q^{-1})q^{\binom{n}{2}-n+2}}{a[n]_{q}\psi _{n}(x)(1-q)}\,_{3}\phi
_{2}\left( 
\begin{array}{c}
q^{-n},x^{-1},\psi _{n}(x) \\ 
0,\psi _{n}(x)q^{-1}%
\end{array}%
;q,a^{-1}qx\right)  \label{ASPTSHR}
\end{equation}%
where $\psi _{n}(x)=((1-q)\vartheta _{n}(x)+1)^{-1}$ and%
\begin{equation*}
\vartheta _{n}(x)=\frac{aq^{n-2}[n]_{q}{\mathcal{C}}_{1,n}(x)}{{\mathcal{D}}%
_{1,n}(x)}-[n-1]_{q}.
\end{equation*}
\end{proposition}

%%%%%%%%%%%%%%%%%%%%%%%%%%%%%%%%%%%%%%%%%%%%%%%%%%%%%%%%%%%%%%%%%%%%%%%%%%%%%%%%%%%%%%%%%%%%%%%%%%%%%%%

\begin{proof}
For $n=0$, a trivial verification shows that (\ref{ASPTSHR}) yields $\mathbb{%
U}_{0}^{(a)}(x;q)=1$. For $n\geq 1$, combining (\ref{ASP}) with (\ref%
{ConexF_I}) and the relations%
\begin{equation}  \label{Ident0}
(q^{1-n};q)_{k}=-\frac{q}{[n]_{q}}([k-1]_{q}-[n-1]_{q})(q^{-n};q)_{k}
\end{equation}
where, in order to improve the compactness and readability of the
expressions involved, we define 
\begin{equation*}
[-1]_{q}:=\frac{1-q^{-1}}{1-q}=-q^{-1},
\end{equation*}
and 
\begin{equation}  \label{Ident}
(q^{-n};q)_{k}=0,\quad n<k,
\end{equation}%
yields%
\begin{equation}  \label{EqProve}
\mathbb{U}_{n}^{(a)}(x;q)=-(-a)^{n-1}q^{\binom{n}{2}-n+2}\frac{{\mathcal{D}}%
_{1,n}(x)}{[n]_{q}}\sum_{k=0}^{n}([k-1]_{q}+\vartheta
_{n}(x))(q^{-n};q)_{k}(x^{-1};q)_{k}\frac{(a^{-1}qx)^{k}}{(q;q)_{k}}.
\end{equation}
More precisely, (\ref{ASP}) together with (\ref{ConexF_I}) yield%
\begin{eqnarray*}
\mathbb{U}_{n}^{(a)}(x;q) &=&(-a)^{n}q^{\binom{n}{2}}{\mathcal{C}}%
_{1,n}(x)\sum_{k=0}^{\infty }\frac{(q^{-n};q)_{k}(x^{-1};q)_{k}}{(q;q)_{k}}%
(a^{-1}qx)^{k} \\
&&+(-a)^{n-1}q^{\binom{n}{2}-n+1}{\mathcal{D}}_{1,n}(x)\sum_{k=0}^{\infty }%
\frac{(q^{1-n};q)_{k}(x^{-1};q)_{k}}{(q;q)_{k}}(a^{-1}qx)^{k}.
\end{eqnarray*}
Taking into account (\ref{Ident}), the previous expression turns into%
\begin{eqnarray*}
\mathbb{U}_{n}^{(a)}(x;q) &=&(-a)^{n}q^{\binom{n}{2}}{\mathcal{C}}%
_{1,n}(x)\sum_{k=0}^{n}\frac{(q^{-n};q)_{k}(x^{-1};q)_{k}}{(q;q)_{k}}%
(a^{-1}qx)^{k} \\
&&+(-a)^{n-1}q^{\binom{n}{2}-n+1}{\mathcal{D}}_{1,n}(x)\sum_{k=0}^{n-1}\frac{%
(q^{1-n};q)_{k}(x^{-1};q)_{k}}{(q;q)_{k}}(a^{-1}qx)^{k},
\end{eqnarray*}%
and (\ref{Ident0}) now leads to%
\begin{eqnarray*}
\mathbb{U}_{n}^{(a)}(x;q) &=&(-a)^{n}q^{\binom{n}{2}}{\mathcal{C}}%
_{1,n}(x)\sum_{k=0}^{n}\frac{(q^{-n};q)_{k}(x^{-1};q)_{k}}{(q;q)_{k}}%
(a^{-1}qx)^{k} \\
&&-(-a)^{n-1}q^{\binom{n}{2}-n+2}\frac{{\mathcal{D}}_{1,n}(x)}{[n]_{q}}%
\sum_{k=0}^{n}([k-1]_{q}-[n-1]_{q})\frac{(q^{-n};q)_{k}(x^{-1};q)_{k}}{%
(q;q)_{k}}(a^{-1}qx)^{k}.
\end{eqnarray*}%
Finally, a rearrangement of the terms in the sums of the previous
expression, leads to (\ref{EqProve}).

On the other hand, after some straightforward calculations we get%
\begin{equation*}
[k-1]_{q}+\vartheta _{n}(x)=\frac{1-\psi _{n}(x)q^{-1}}{\psi _{n}(x)(1-q)}%
\frac{(\psi _{n}(x);q)_{k}}{(\psi _{n}(x)q^{-1};q)_{k}}.
\end{equation*}%
Therefore%
\begin{equation*}
\mathbb{U}_{n}^{(a)}(x;q)=-(-a)^{n-1}\frac{{\mathcal{D}}_{1,n}(x)(1-\psi
_{n}(x)q^{-1})q^{\binom{n}{2}-n+2}}{[n]_{q}\psi _{n}(x)(1-q)}\sum_{k=0}^{n}%
\frac{(q^{-n};q)_{k}(x^{-1};q)_{k}(\psi _{n}(x);q)_{k}}{(\psi
_{n}(x)q^{-1};q)_{k}}\frac{(a^{-1}qx)^{k}}{(q;q)_{k}}
\end{equation*}
which coincides with (\ref{ASPTSHR}). This completes the proof.
\end{proof}

%%%%%%%%%%%%%%%%%%%%%%%%%%%%%%%%%%%%%%%%%%%%%%%%%%%%%%%%%%%%%%%%%%%%%%%%%%%%%%%%%%%%%%%%%%%%%%%%%%%%%%%

%%%%%%%%%%%%%%%%%%%%%%%%%%%%%%%%%%%%%%%%%%%%%%%%%%%%%%%%%%%%%%%%%%%%%%%%%%%%%%%%%%%%%%%%%%%%%%%%%%%%%%%

\begin{remark}
Notice that one recovers (\ref{ASP}) from (\ref{ASPTSHR}) after $\lambda
\rightarrow 0$ and $\mu \rightarrow 0$. One also recovers \cite[(23), p. 13]%
{CS-JDE18} for $j=1$, and $\lambda ,\mu >0$ in (\ref{ASPTSHR}).
\end{remark}

%%%%%%%%%%%%%%%%%%%%%%%%%%%%%%%%%%%%%%%%%%%%%%%%%%%%%%%%%%%%%%%%%%%%%%%%%%%%%%%%%%%%%%%%%%%%%%%%%%%%%%%

%%%%%%%%%%%%%%%%%%%%%%%%%%%%%%%%%%%%%%%%%%%%%%%%%%%%%%%%%%%%%%%%%%%%%%%%%%%%%%%%%%%%%%%%%%%%%%%%%%%%%%%

\section{Ladder operators and a three term recurrence formula}

\label{S4-LOp}

%%%%%%%%%%%%%%%%%%%%%%%%%%%%%%%%%%%%%%%%%%%%%%%%%%%%%%%%%%%%%%%%%%%%%%%%%%%%%%%%%%%%%%%%%%%%%%%%%%%%%%%

%%%%%%%%%%%%%%%%%%%%%%%%%%%%%%%%%%%%%%%%%%%%%%%%%%%%%%%%%%%%%%%%%%%%%%%%%%%%%%%%%%%%%%%%%%%%%%%%%%%%%%%

In this section we find several structure relations associated to $\{\mathbb{%
U}_{n}^{(a)}(x;q)\}_{n\geq 0}$. It is worth mentioning that such relations
can be grouped in two depending on nature of the action of the $q-$%
derivative involved in the relation (see Theorem~\ref{sr_theorem}). In two
of them, such $q-$derivative is constructed by means of a $q-$dilation
operator ($\ell=-1$) whether in the other two, a $q-$contraction operator
determines the $q-$derivative ($\ell=1$). The \textit{ladder (creation and
annihilation) operators} are obtained in Proposition~\ref{S4-PropLadder}, as
well as the three-term recurrence relations of Theorem~\ref{S4-Theor3TRR-RC}%
, satisfied by $\{\mathbb{U}_{n}^{(a)}(x;q)\}_{n\geq 0}$. These two results
are also stated in terms of the duality provided by the choice of a $q-$%
dilating or $q-$contracting derivation.

%%%%%%%%%%%%%%%%%%%%%%%%%%%%%%%%%%%%%%%%%%%%%%%%%%%%%%%%%%%%%%%%%%%%%%%%%%%%%%%%%%%%%%%%%%%%%%%%%%%%%%%

%%%%%%%%%%%%%%%%%%%%%%%%%%%%%%%%%%%%%%%%%%%%%%%%%%%%%%%%%%%%%%%%%%%%%%%%%%%%%%%%%%%%%%%%%%%%%%%%%%%%%%%

The structure relations stated in Theorem~\ref{sr_theorem} lean on the
following result.

\begin{lemma}
\label{LemmaLadder} Let $\{\mathbb{U}_{n}^{(a)}(x;q)\}_{n\geq 0}$ be the
sequence of Al-Salam--Carlitz I-Sobolev type polynomials of degree $n$.
Then, following statements hold, for $\ell =-1,1$, 
\begin{equation}
\sigma _{\ell }(x){\mathscr D}_{q^{\ell }}\mathbb{U}_{n}^{(a)}(x;q)={%
\mathcal{E}}_{1+2\delta _{\ell ,1},n}(x)\cdot U_{n}^{(a)}(x;q)+{\mathcal{F}}%
_{1+2\delta _{\ell ,1},n}(x)\cdot U_{n-1}^{(a)}(x;q),  \label{CFD1}
\end{equation}%
and 
\begin{equation}
\sigma _{\ell }(x){\mathscr D}_{q^{\ell }}\mathbb{U}_{n-1}^{(a)}(x;q)={%
\mathcal{E}}_{2+2\delta _{\ell ,1},n}(x)\cdot U_{n}^{(a)}(x;q)+{\mathcal{F}}%
_{2+2\delta _{\ell ,1},n}(x)\cdot U_{n-1}^{(a)}(x;q),  \label{CFD2}
\end{equation}%
where $\sigma _{\ell }(x)=\sigma (x)$ for $\ell =-1$ and $\sigma _{\ell
}(x)=1$ otherwise. Moreover, 
\begin{equation*}
{\mathcal{E}}_{1,n}(x)=(\overline{\alpha }_{n}(x-\beta _{n})+\overline{\beta 
}_{n}){\mathcal{C}}_{1,n}(q^{-1}x)+\sigma (x){\mathscr D}_{q^{-1}}{\mathcal{C%
}}_{1,n}(x)+(\overline{\alpha }_{n-1}-\overline{\gamma }_{n-1}\gamma
_{n-1}^{-1}){\mathcal{D}}_{1,n}(q^{-1}x),
\end{equation*}%
\begin{equation*}
{\mathcal{F}}_{1,n}(x)=(\overline{\gamma }_{n}-\overline{\alpha }_{n}\gamma
_{n}){\mathcal{C}}_{1,n}(q^{-1}x)+\sigma (x){\mathscr D}_{q^{-1}}{\mathscr D}%
_{1,n}(x)+(\overline{\beta }_{n-1}+\overline{\gamma }_{n-1}\gamma
_{n-1}^{-1}(x-\beta _{n-1})){\mathcal{D}}_{1,n}(q^{-1}x),
\end{equation*}%
\begin{equation*}
{\mathcal{E}}_{3,n}(x)={\mathscr D}_{q}{\mathcal{C}}_{1,n}(x)-[n-1]_{q}%
\gamma _{n-1}^{-1}{\mathcal{D}}_{1,n}(qx),
\end{equation*}%
\begin{equation*}
{\mathcal{F}}_{3,n}(x)=[n]_{q}{\mathcal{C}}_{1,n}(qx)+[n-1]_{q}\gamma
_{n-1}^{-1}(x-\beta _{n-1}){\mathcal{D}}_{1,n}(qx)+{\mathscr D}_{q}{\mathcal{%
D}}_{1,n}(x),
\end{equation*}%
\begin{equation*}
{\mathcal{E}}_{2+2\delta _{\ell ,1},n}(x)=-\frac{{\mathcal{F}}_{1+2\delta
_{\ell ,1},n-1}(x)}{\gamma _{n-1}},
\end{equation*}%
and 
\begin{equation*}
{\mathcal{F}}_{2+2\delta _{\ell ,1},n}(x)={\mathcal{E}}_{1+2\delta _{\ell
,1},n-1}(x)+{\mathcal{E}}_{2+2\delta _{\ell ,1},n}(x)(\beta _{n-1}-x).
\end{equation*}%
Here, $\delta _{m,n}$ denote the Kronecker delta function.
\end{lemma}

\begin{proof}
It is a direct consequence of the connection formulas (\ref{ConexF_I})--(\ref%
{ConexF_IV}), the three-term recurrence relation (\ref{ReR}) satisfied by $%
\{U_{n}^{(a)}(x;q)\}_{n\geq 0}$, and the structure relation (\ref{StruR}).
To be more precise, applying the $q$-derivative operator ${\mathscr D}%
_{q^{\ell }}$ to (\ref{ConexF_I}) for $\ell =-1,1$, together with the
property (\ref{prodqD}) yields 
\begin{eqnarray*}
{\mathscr D}_{q^{\ell }}\mathbb{U}_{n}^{(a)}(x;q) &=&{\mathcal{C}}%
_{1,n}(q^{\ell }x)\cdot {\mathscr D}_{q^{\ell
}}U_{n}^{(a)}(x;q)+U_{n}^{(a)}(x;q)\cdot {\mathscr D}_{q^{\ell }}{\mathcal{C}%
}_{1,n}(x) \\
&&+{\mathcal{D}}_{1,n}(q^{\ell }x)\cdot {\mathscr D}_{q^{\ell
}}U_{n-1}^{(a)}(x;q)+U_{n-1}^{(a)}(x;q)\cdot {\mathscr D}_{q^{\ell }}{%
\mathcal{D}}_{1,n}(x).
\end{eqnarray*}%
Thus, multiplying the above expression by $\sigma _{\ell }(x)$ for $\ell
=-1,1 $, and next combining (\ref{StruR}) with (\ref{FSop}) and (\ref{ReR}),
we deduce (\ref{CFD1}). Finally, shifting the index in (\ref{CFD1}) as $%
n\rightarrow n-1$ and using the recurrence relation (\ref{ReR}) we get (\ref%
{CFD2}). This completes the proof.
\end{proof}

%%%%%%%%%%%%%%%%%%%%%%%%%%%%%%%%%%%%%%%%%%%%%%%%%%%%%%%%%%%%%%%%%%%%%%%%%%%%%%%%%%%%%%%%%%%%%%%%%%%%%%%

%%%%%%%%%%%%%%%%%%%%%%%%%%%%%%%%%%%%%%%%%%%%%%%%%%%%%%%%%%%%%%%%%%%%%%%%%%%%%%%%%%%%%%%%%%%%%%%%%%%%%%%

As a direct consequence of the previous result, we next obtain the following
structure relations for the Al-Salam--Carlitz I-Sobolev type polynomials of
higher order.

%%%%%%%%%%%%%%%%%%%%%%%%%%%%%%%%%%%%%%%%%%%%%%%%%%%%%%%%%%%%%%%%%%%%%%%%%%%%%%%%%%%%%%%%%%%%%%%%%%%%%%%

%%%%%%%%%%%%%%%%%%%%%%%%%%%%%%%%%%%%%%%%%%%%%%%%%%%%%%%%%%%%%%%%%%%%%%%%%%%%%%%%%%%%%%%%%%%%%%%%%%%%%%%

\begin{theorem}
\label{sr_theorem} The Al-Salam--Carlitz I-Sobolev type polynomials of high
order $\{\mathbb{U}_{n}^{(a)}(x;q)\}_{n\geq 0}$ satisfy the following
structure relations for $\ell =-1,1$,%
\begin{equation}
\Theta _{\ell ,n}(x)\cdot {\mathscr D}_{q^{\ell }} \mathbb{U}%
_{n}^{(a)}(x;q)=\Xi _{2,1+2\delta _{\ell ,1},n}(x)\cdot \mathbb{U}%
_{n}^{(a)}(x;q)+\Xi _{1,1+2\delta _{\ell ,1},n}(x)\cdot \mathbb{U}%
_{n-1}^{(a)}(x;q),  \label{DES1}
\end{equation}%
and 
\begin{equation}
\Theta _{\ell ,n}(x)\cdot {\mathscr D}_{q^{\ell }}\mathbb{U}%
_{n-1}^{(a)}(x;q)=\Xi _{2,2+2\delta _{\ell ,1},n}(x)\cdot \mathbb{U}%
_{n}^{(a)}(x;q)+\Xi _{1,2+2\delta _{\ell ,1},n}(x)\cdot \mathbb{U}%
_{n-1}^{(a)}(x;q),  \label{DES2}
\end{equation}%
where $\Theta _{\ell ,n}(x)=\sigma _{\ell }(x)\det (B_{n}(x))$ and 
\begin{equation*}
\Xi _{i,k,n}(x)=(-1)^{i}\det 
\begin{pmatrix}
{\mathcal{E}}_{k,n}(x) & {\mathcal{C}}_{i,n}(x) \\ 
{\mathcal{F}}_{k,n}(x) & {\mathcal{D}}_{i,n}(x)%
\end{pmatrix}%
,\quad i=1,2,\quad k=1,2,3,4.
\end{equation*}
\end{theorem}

\begin{proof}
The result follows in a straightforward way from (\ref{ConexF_III})--(\ref%
{ConexF_IV}), together with the application of the previous Lemma \ref%
{LemmaLadder}.
\end{proof}

%%%%%%%%%%%%%%%%%%%%%%%%%%%%%%%%%%%%%%%%%%%%%%%%%%%%%%%%%%%%%%%%%%%%%%%%%%%%%%%%%%%%%%%%%%%%%%%%%%%%%%%

%%%%%%%%%%%%%%%%%%%%%%%%%%%%%%%%%%%%%%%%%%%%%%%%%%%%%%%%%%%%%%%%%%%%%%%%%%%%%%%%%%%%%%%%%%%%%%%%%%%%%%%

In the next result, we provide ladder operators associated to the
Al-Salam--Carlitz I-Sobolev type polynomials. Its proof is quite involved,
leaning on the use of Theorem \ref{sr_theorem}, and following the same
technique as in \cite{GGH-M20}.

%%%%%%%%%%%%%%%%%%%%%%%%%%%%%%%%%%%%%%%%%%%%%%%%%%%%%%%%%%%%%%%%%%%%%%%%%%%%%%%%%%%%%%%%%%%%%%%%%%%%%%%

%%%%%%%%%%%%%%%%%%%%%%%%%%%%%%%%%%%%%%%%%%%%%%%%%%%%%%%%%%%%%%%%%%%%%%%%%%%%%%%%%%%%%%%%%%%%%%%%%%%%%%%

\begin{proposition}
\label{S4-PropLadder} Let $\{\mathbb{U}_{n}^{(a)}(x;q)\}_{n\geq 0}$ be the
sequence of Al-Salam--Carlitz I-Sobolev type polynomials defined by (\ref%
{ASPTSHR}), and let $I$ be the identity operator. Then, the ladder
(destruction and creation) operators ${\mathfrak{a}_{\ell }}$ and ${%
\mathfrak{a}_{\ell }^{\dagger }}$, respectively, are defined by%
\begin{equation}
{\mathfrak{a}_{\ell }}=\Theta _{\ell ,n}(x){\mathscr D}_{q^{\ell }}-\Xi
_{2,1+2\delta _{\ell ,1},n}(x)I,  \label{eqDO}
\end{equation}%
\begin{equation}
{\mathfrak{a}_{\ell }^{\dagger }}=\Theta _{\ell ,n}(x){\mathscr D}_{q^{\ell
}}-\Xi _{1,2+2\delta _{\ell ,1},n}(x)I,  \label{eqCO}
\end{equation}%
which verify 
\begin{equation}
{\mathfrak{a}_{\ell }}\left( \mathbb{U}_{n}^{(a)}(x;q)\right) =\Xi
_{1,1+2\delta _{\ell ,1},n}(x)\mathbb{U}_{n-1}^{(a)}(x;q),  \label{DO}
\end{equation}%
\begin{equation}
{\mathfrak{a}_{\ell }^{\dagger }}\left( \mathbb{U}_{n-1}^{(a)}(x;q)\right)
=\Xi _{2,2+2\delta _{\ell ,1},n}(x)\mathbb{U}_{n}^{(a)}(x;q),  \label{CO}
\end{equation}%
where $\ell =-1,1$.
\end{proposition}

%%%%%%%%%%%%%%%%%%%%%%%%%%%%%%%%%%%%%%%%%%%%%%%%%%%%%%%%%%%%%%%%%%%%%%%%%%%%%%%%%%%%%%%%%%%%%%%%%%%%%%%

%%%%%%%%%%%%%%%%%%%%%%%%%%%%%%%%%%%%%%%%%%%%%%%%%%%%%%%%%%%%%%%%%%%%%%%%%%%%%%%%%%%%%%%%%%%%%%%%%%%%%%%

We complete this Section with a straightforward application of the above
ladder operators ${\mathfrak{a}_{\ell }}$ and ${\mathfrak{a}_{\ell
}^{\dagger }}$. We use these operators to obtain two versions (one for $\ell
=-1$ and other for $\ell =1$) of certain three term recurrence formula with
rational coefficients which provides $\mathbb{U}_{n+1}^{(a)}(x;q)$ in terms
of the former two consecutive polynomials $\mathbb{U}_{n}^{(a)}(x;q)$ and $%
\mathbb{U}_{n-1}^{(a)}(x;q)$. We recall the importance of such recurrence
formula, which allows to give further properties of the family of orthogonal
polynomials, as described in classical references such as~\cite{Chi-78}. The
proof of the next result can be followed from the aforementioned technique,
which has been recently generalized in \cite[p. 8]{GGH-M20}. However, we
have decided to include it below for the sake of completeness.

\begin{theorem}
\label{S4-Theor3TRR-RC}The Al-Salam--Carlitz I-Sobolev type polynomials of
high order $\{\mathbb{U}_{n}^{(a)}(x;q)\}_{n\geq 0}$ satisfy the following
three-term recurrence relations for $\ell =-1,1$, 
\begin{equation}
\alpha _{\ell ,n}(x)\mathbb{U}_{n+1}^{(a)}(x;q)=\beta _{\ell ,n}(x)\mathbb{U}%
_{n}^{(a)}(x;q)+\gamma _{\ell ,n}(x)\mathbb{U}_{n-1}^{(a)}(x;q),
\label{3TRR-RC}
\end{equation}%
where 
\begin{equation*}
\alpha _{\ell ,n}(x)=\Theta _{\ell ,n}(x)\Xi _{2,2+2\delta _{\ell
,1},n+1}(x),
\end{equation*}%
\begin{equation*}
\beta _{\ell ,n}(x)=\Theta _{\ell ,n+1}(x)\Xi _{2,1+2\delta _{\ell
,1},n}(x)-\Theta _{\ell ,n}(x)\Xi _{1,2+2\delta _{\ell ,1},n+1}(x),
\end{equation*}%
and 
\begin{equation*}
\gamma _{\ell ,n}(x)=\Theta _{\ell ,n+1}(x)\Xi _{1,1+2\delta _{\ell
,1},n}(x).
\end{equation*}
\end{theorem}

%%%%%%%%%%%%%%%%%%%%%%%%%%%%%%%%%%%%%%%%%%%%%%%%%%%%%%%%%%%%%%%%%%%%%%%%%%%%%%%%%%%%%%%%%%%%%%%%%%%%%%%

%%%%%%%%%%%%%%%%%%%%%%%%%%%%%%%%%%%%%%%%%%%%%%%%%%%%%%%%%%%%%%%%%%%%%%%%%%%%%%%%%%%%%%%%%%%%%%%%%%%%%%%

\begin{proof}
Shifting the index in (\ref{DES2}) as $n\rightarrow n+1$, yields 
\begin{equation*}
\Theta _{\ell ,n+1}(x){\mathscr D}_{q^{\ell }}\mathbb{U}_{n}^{(a)}(x;q)=\Xi
_{2,2+2\delta _{\ell ,1},n+1}(x)\mathbb{U}_{n+1}^{(a)}(x;q)+\Xi
_{1,2+2\delta _{\ell ,1},n+1}(x)\mathbb{U}_{n}^{(a)}(x;q).
\end{equation*}%
Next, multiplying the above expression by $-\Theta _{\ell ,n}(x)$, and
multiplying (\ref{DES1}) by $\Theta _{\ell ,n+1}(x)$, adding and simplifying
the resulting equations, we obtain (\ref{3TRR-RC}). This completes the proof.
\end{proof}

\begin{remark}
Notice that (\ref{3TRR-RC}) becomes (\ref{ReR}) when $\lambda=\mu=0$.
\end{remark}

%%%%%%%%%%%%%%%%%%%%%%%%%%%%%%%%%%%%%%%%%%%%%%%%%%%%%%%%%%%%%%%%%%%%%%%%%%%%%%%%%%%%%%%%%%%%%%%%%%%%%%%

%%%%%%%%%%%%%%%%%%%%%%%%%%%%%%%%%%%%%%%%%%%%%%%%%%%%%%%%%%%%%%%%%%%%%%%%%%%%%%%%%%%%%%%%%%%%%%%%%%%%%%%

\section{Holonomic second order $q$-difference equations}

\label{S5-LdE}

%%%%%%%%%%%%%%%%%%%%%%%%%%%%%%%%%%%%%%%%%%%%%%%%%%%%%%%%%%%%%%%%%%%%%%%%%%%%%%%%%%%%%%%%%%%%%%%%%%%%%%%

%%%%%%%%%%%%%%%%%%%%%%%%%%%%%%%%%%%%%%%%%%%%%%%%%%%%%%%%%%%%%%%%%%%%%%%%%%%%%%%%%%%%%%%%%%%%%%%%%%%%%%%

In this section, we find up to four different $q$-difference equations of
second order that $\mathbb{U}_{n}^{(a)}(x;q)$ satisfies. It is worth noting
that these $q$-difference equations are no longer classical, in the sense
that all their polynomial coefficients depend on $n$, so we have different
coefficients for every different degree $n$ that we consider. When dealing
with differential equations, these kind of equations are known in the
literature as \textit{Holonomic second order differential equations} (see,
for example \cite{GGH-M20}), so by natural extension we refer to them as 
\textit{Second order linear }$\mathit{q}$\textit{-difference equations}.

It is also worth remarking the different nature of these four equations. The
four of them are linear second order difference equations. However, in two
of them (see Proposition~\ref{S5-Proposition51}), a unique difference
operator appears. We also point out the appearance of four regular singular
points in one of these two equations, more precisely for the choice $\ell=-1$%
. Such points are $1,a,q$ and $aq$, which appear to be intimately related to
the problem. The two second difference equations (see Proposition~\ref%
{S5-Proposition52}) an analogous disquisition can be made concerning the
singular points, which are now the points $1,a,q^{-1}$ and $aq^{-1}$.
Moreover, the two last equations involve the two $q-$difference operators
previously mentioned, in contrast to the two first ones.

%%%%%%%%%%%%%%%%%%%%%%%%%%%%%%%%%%%%%%%%%%%%%%%%%%%%%%%%%%%%%%%%%%%%%%%%%%%%%%%%%%%%%%%%%%%%%%%%%%%%%%%

\begin{proposition}[2nd order holonomic eq. I]
\label{S5-Proposition51} Let $\{\mathbb{U}_{n}^{(a)}(x;q)\}_{n\geq 0}$ be
the sequence of Al-Salam--Carlitz I-Sobolev type polynomials defined by (\ref%
{ASPTSHR}). Then, the following statement holds, for $\ell =-1,1$, 
\begin{equation}
{\mathcal{R}}_{\ell ,n}(x){\mathscr D}_{q^{\ell }}^{2}\mathbb{U}%
_{n}^{(a)}(x;q)+{\mathcal{S}}_{\ell ,n}(x){\mathscr D}_{q^{\ell }}\mathbb{U}%
_{n}^{(a)}(x;q)+{\mathcal{T}}_{\ell ,n}(x)\mathbb{U}_{n}^{(a)}(x;q)=0,\quad
n\geq 0,  \label{SDC1}
\end{equation}%
where 
\begin{equation*}
{\mathcal{R}}_{\ell ,n}(x)=\Theta _{\ell ,n}(x)\Theta _{\ell ,n}(q^{\ell }x),
\end{equation*}%
\begin{eqnarray*}
{\mathcal{S}}_{\ell ,n}(x) &=&\Theta _{\ell ,n}(x)\left[ {\mathscr D}%
_{q^{\ell }}\Theta _{\ell ,n}(x)-\Xi _{2,1+2\delta _{\ell ,1},n}(q^{\ell
}x)-\Xi _{1,2+2\delta _{\ell ,1},n}(x)\right] \\
&&\quad -\frac{\Theta _{\ell ,n}(x)\left[ \Theta _{\ell ,n}(x)+(q^{\ell
}-1)x\,\Xi _{1,2+2\delta _{\ell ,1},n}(x)\right] {\mathscr D}_{q^{\ell }}\Xi
_{1,1+2\delta _{\ell ,1},n}(x)}{\Xi _{1,1+2\delta _{\ell ,1},n}(x)},
\end{eqnarray*}%
and 
\begin{eqnarray*}
{\mathcal{T}}_{\ell ,n}(x) &=&\Xi _{1,2+2\delta _{\ell ,1},n}(x)\Xi
_{2,1+2\delta _{\ell ,1},n}(x)-\Theta _{\ell ,n}(x){\mathscr D}_{q^{\ell
}}\Xi _{2,1+2\delta _{\ell ,1},n}(x) \\
&&\quad +\frac{\Xi _{2,1+2\delta _{\ell ,1},n}(x)\left[ \Theta _{\ell
,n}(x)+(q^{\ell }-1)x\,\Xi _{1,2+2\delta _{\ell ,1},n}(x)\right] {\mathscr D}%
_{q^{\ell }}\Xi _{1,1+2\delta _{\ell ,1},n}(x)}{\Xi _{1,1+2\delta _{\ell
,1},n}(x)} \\
&&\quad \quad -\Xi _{1,1+2\delta _{\ell ,1},n}(q^{\ell }x)\Xi _{2,2+2\delta
_{\ell ,1},n}(x).
\end{eqnarray*}
\end{proposition}

%%%%%%%%%%%%%%%%%%%%%%%%%%%%%%%%%%%%%%%%%%%%%%%%%%%%%%%%%%%%%%%%%%%%%%%%%%%%%%%%%%%%%%%%%%%%%%%%%%%%%%%

%%%%%%%%%%%%%%%%%%%%%%%%%%%%%%%%%%%%%%%%%%%%%%%%%%%%%%%%%%%%%%%%%%%%%%%%%%%%%%%%%%%%%%%%%%%%%%%%%%%%%%%

\begin{proof}
In fact, from (\ref{DO}) we have%
\begin{equation}
{\mathfrak{a}_{\ell }^{\dagger }}\left[ {\mathfrak{a}_{\ell }}\left( \mathbb{%
U}_{n}^{(a)}(x;q)\right) \right] ={\mathfrak{a}_{\ell }^{\dagger }}\left[
\Xi _{1,1+2\delta _{\ell ,1},n}(x)\mathbb{U}_{n-1}^{(a)}(x;q)\right] .
\label{eqproof1}
\end{equation}%
Then, applying (\ref{eqDO}) and (\ref{eqCO}) to left hand member of (\ref%
{eqproof1}) we have%
\begin{equation*}
{\mathfrak{a}_{\ell }^{\dagger }}\left[ {\mathfrak{a}_{\ell }}\left( \mathbb{%
U}_{n}^{(a)}(x;q)\right) \right] ={\mathfrak{a}_{\ell }^{\dagger }}\left[
\Theta _{\ell ,n}(x){\mathscr D}_{q^{\ell }}\mathbb{U}_{n}^{(a)}(x;q)-\Xi
_{2,1+2\delta _{\ell ,1},n}(x)\mathbb{U}_{n}^{(a)}(x;q)\right]
\end{equation*}%
\begin{equation*}
=\Theta _{\ell ,n}(x){\mathscr D}_{q^{\ell }}\left[ \Theta _{\ell ,n}(x){%
\mathscr D}_{q^{\ell }}\mathbb{U}_{n}^{(a)}(x;q)-\Xi _{2,1+2\delta _{\ell
,1},n}(x)\mathbb{U}_{n}^{(a)}(x;q)\right]
\end{equation*}%
\begin{equation*}
-\Xi _{1,2+2\delta _{\ell ,1},n}(x)\Theta _{\ell ,n}(x){\mathscr D}_{q^{\ell
}}\mathbb{U}_{n}^{(a)}(x;q)+\Xi _{1,2+2\delta _{\ell ,1},n}(x)\Xi
_{2,1+2\delta _{\ell ,1},n}(x)\mathbb{U}_{n}^{(a)}(x;q).
\end{equation*}%
Next, using (\ref{prodqD}) we arrive%
\begin{equation*}
{\mathfrak{a}_{\ell }^{\dagger }}\left[ {\mathfrak{a}_{\ell }}\left( \mathbb{%
U}_{n}^{(a)}(x;q)\right) \right] =\Theta _{\ell ,n}(x)\Theta _{\ell
,n}(q^{\ell }x){\mathscr D}_{q^{\ell }}^{2}\mathbb{U}_{n}^{(a)}(x;q)
\end{equation*}%
\begin{equation}
+\Theta _{\ell ,n}(x)\left[ {\mathscr D}_{q^{\ell }}\Theta _{\ell ,n}(x)-\Xi
_{2,1+2\delta _{\ell ,1},n}(q^{\ell }x)-\Xi _{1,2+2\delta _{\ell ,1},n}(x)%
\right] {\mathscr D}_{q^{\ell }}\mathbb{U}_{n}^{(a)}(x;q)  \label{1part}
\end{equation}%
\begin{equation*}
+\left[ \Xi _{1,2+2\delta _{\ell ,1},n}(x)\Xi _{2,1+2\delta _{\ell
,1},n}(x)-\Theta _{\ell ,n}(x){\mathscr D}_{q^{\ell }}\Xi _{2,1+2\delta
_{\ell ,1},n}(x)\right] \mathbb{U}_{n}^{(a)}(x;q).
\end{equation*}
Applying (\ref{eqCO}) to the right hand side of (\ref{eqproof1}) we have%
\begin{eqnarray*}
{\mathfrak{a}_{\ell }^{\dagger }}\left[ \Xi _{1,1+2\delta _{\ell ,1},n}(x)%
\mathbb{U}_{n-1}^{(a)}(x;q)\right] &=&\Theta _{\ell ,n}(x){\mathscr D}%
_{q^{\ell }}[\Xi _{1,1+2\delta _{\ell ,1},n}(x)\mathbb{U}_{n-1}^{(a)}(x;q)]
\\
&&-\Xi _{1,2+2\delta _{\ell ,1},n}(x)\Xi _{1,1+2\delta _{\ell ,1},n}(x)%
\mathbb{U}_{n-1}^{(a)}(x;q).
\end{eqnarray*}%
From property (\ref{prodqD}) we deduce%
\begin{eqnarray*}
{\mathfrak{a}_{\ell }^{\dagger }}\left[ \Xi _{1,1+2\delta _{\ell ,1},n}(x)%
\mathbb{U}_{n-1}^{(a)}(x;q)\right] &=&\Xi _{1,1+2\delta _{\ell
,1},n}(q^{\ell }x)\Theta _{\ell ,n}(x){\mathscr D}_{q^{\ell }}\mathbb{U}%
_{n-1}^{(a)}(x;q) \\
&&+[\Theta _{\ell ,n}(x){\mathscr D}_{q^{\ell }}\Xi _{1,1+2\delta _{\ell
,1},n}(x)-\Xi _{1,2+2\delta _{\ell ,1},n}(x)\Xi _{1,1+2\delta _{\ell
,1},n}(x)]\mathbb{U}_{n-1}^{(a)}(x;q).
\end{eqnarray*}%
Therefore, from (\ref{DES2}) we get%
\begin{eqnarray*}
{\mathfrak{a}_{\ell }^{\dagger }}\left[ \Xi _{1,1+2\delta _{\ell ,1},n}(x)%
\mathbb{U}_{n-1}^{(a)}(x;q)\right] &=&\Xi _{1,1+2\delta _{\ell
,1},n}(q^{\ell }x)\Xi _{2,2+2\delta _{\ell ,1},n}(x)\mathbb{U}_{n}^{(a)}(x;q)
\\
&&+\Xi _{1,1+2\delta _{\ell ,1},n}(q^{\ell }x)\Xi _{1,2+2\delta _{\ell
,1},n}(x)\mathbb{U}_{n-1}^{(a)}(x;q) \\
&&+[\Theta _{\ell ,n}(x){\mathscr D}_{q^{\ell }}\Xi _{1,1+2\delta _{\ell
,1},n}(x)-\Xi _{1,2+2\delta _{\ell ,1},n}(x)\Xi _{1,1+2\delta _{\ell
,1},n}(x)]\mathbb{U}_{n-1}^{(a)}(x;q).
\end{eqnarray*}
Property (\ref{q-derivative}) leads to%
\begin{eqnarray*}
{\mathfrak{a}_{\ell }^{\dagger }}\left[ \Xi _{1,1+2\delta _{\ell ,1},n}(x)%
\mathbb{U}_{n-1}^{(a)}(x;q)\right] &=&\Xi _{1,1+2\delta _{\ell
,1},n}(q^{\ell }x)\Xi _{2,2+2\delta _{\ell ,1},n}(x)\mathbb{U}_{n}^{(a)}(x;q)
\\
&&+[\Theta _{\ell ,n}(x)+(q-1)x\Xi _{1,2+2\delta _{\ell ,1},n}(x)]{\mathscr D%
}_{q^{\ell }}\Xi _{1,1+2\delta _{\ell ,1},n}(x)\mathbb{U}_{n-1}^{(a)}(x;q).
\end{eqnarray*}
Next, we apply \eqref{DES1} to arrive at 
\begin{equation*}
\mathbb{U}_{n-1}^{(a)}(x;q)=\frac{\Theta _{\ell ,n}(x)}{\Xi _{1,1+2\delta
_{\ell ,1},n}(x)}{\mathscr D}_{q^{\ell }}\mathbb{U}_{n}^{(a)}(x;q)-\frac{\Xi
_{2,1+2\delta _{\ell ,1},n}(x)}{\Xi _{1,1+2\delta _{\ell ,1},n}(x)}\mathbb{U}%
_{n}^{(a)}(x;q),
\end{equation*}
and therefore 
\begin{equation*}
{\mathfrak{a}_{\ell }^{\dagger }}\left[ \Xi _{1,1+2\delta _{\ell ,1},n}(x)%
\mathbb{U}_{n-1}^{(a)}(x;q)\right]
\end{equation*}%
\begin{equation}
=\frac{\Theta _{\ell ,n}(x)\left[ \Theta _{\ell ,n}(x)+(q-1)x\,\Xi
_{1,2+2\delta _{\ell ,1},n}(x)\right] {\mathscr D}_{q^{\ell }}\Xi
_{1,1+2\delta _{\ell ,1},n}(x)}{\Xi _{1,1+2\delta _{\ell ,1},n}(x)}{\mathscr %
D}_{q^{\ell }}\mathbb{U}_{n}^{(a)}(x;q)  \label{2part}
\end{equation}%
\begin{equation*}
+\Big[\Xi _{1,1+2\delta _{\ell ,1},n}(q^{\ell }x)\Xi _{2,2+2\delta _{\ell
,1},n}(x)
\end{equation*}%
\begin{equation*}
-\frac{\Xi _{2,1+2\delta _{\ell ,1},n}(x)\left[ \Theta _{\ell
,n}(x)+(q-1)x\,\Xi _{1,2+2\delta _{\ell ,1},n}(x)\right] {\mathscr D}%
_{q^{\ell }}\Xi _{1,1+2\delta _{\ell ,1},n}(x)}{\Xi _{1,1+2\delta _{\ell
,1},n}(x)}\Big]\mathbb{U}_{n}^{(a)}(x;q).
\end{equation*}%
Finally, equaling (\ref{1part}) and (\ref{2part}) we arrived to the desired
result.

%%%%%%%%%%%%%%%%%%%
\end{proof}

%%%%%%%%%%%%%%%%%%%%%%%%%%%%%%%%%%%%%%%%%%%%%%%%%%%%%%%%%%%%%%%%%%%%%%%%%%%%%%%%%%%%%%%%%%%%%%%%%%%%%%%

%%%%%%%%%%%%%%%%%%%%%%%%%%%%%%%%%%%%%%%%%%%%%%%%%%%%%%%%%%%%%%%%%%%%%%%%%%%%%%%%%%%%%%%%%%%%%%%%%%%%%%%

\begin{proposition}[2nd order holonomic eq. II]
\label{S5-Proposition52} Let $\{\mathbb{U}_{n}^{(a)}(x;q)\}_{n\geq 0}$ be
the sequence of Al-Salam--Carlitz I-Sobolev type polynomials defined by (\ref%
{ASPTSHR}). Then, the following statement holds, for $\ell =-1,1$, 
\begin{equation}
\overline{{\mathcal{R}}}_{\ell ,n}(x)\mathbb{D}_{q^{\ell }}^{2}\mathbb{U}%
_{n}^{(a)}(x;q)+\overline{{\mathcal{S}}}_{\ell ,n}(x){\mathscr D}_{q^{-\ell
}}\mathbb{U}_{n}^{(a)}(x;q)+\overline{{\mathcal{T}}}_{\ell ,n}(x)\mathbb{U}%
_{n}^{(a)}(x;q)=0,\quad n\geq 0,  \label{SDC2}
\end{equation}%
where 
\begin{equation*}
\mathbb{D}_{q^{\ell }}^{2}=%
\begin{cases}
\displaystyle{\mathscr D}_{q}{\mathscr D}_{q^{-1}}, & \ell =-1, \\ 
&  \\ 
\displaystyle{\mathscr D}_{q^{-1}}{\mathscr D}_{q}, & \ell =1,%
\end{cases}%
\end{equation*}%
and 
\begin{equation*}
\overline{{\mathcal{R}}}_{\ell ,n}(x)={\mathcal{R}}_{\ell ,n}(q^{-\ell
}x),\quad \overline{{\mathcal{S}}}_{\ell ,n}(x)={\mathcal{S}}_{\ell
,n}(q^{-\ell }x)+(q^{-\ell }-1)x{\mathcal{T}}_{\ell ,n}(q^{-\ell }x),\quad 
\overline{{\mathcal{T}}}_{\ell ,n}(x)={\mathcal{T}}_{\ell ,n}(q^{-\ell }x).
\end{equation*}
\end{proposition}

%%%%%%%%%%%%%%%%%%%%%%%%%%%%%%%%%%%%%%%%%%%%%%%%%%%%%%%%%%%%%%%%%%%%%%%%%%%%%%%%%%%%%%%%%%%%%%%%%%%%%%%

%%%%%%%%%%%%%%%%%%%%%%%%%%%%%%%%%%%%%%%%%%%%%%%%%%%%%%%%%%%%%%%%%%%%%%%%%%%%%%%%%%%%%%%%%%%%%%%%%%%%%%%

\begin{proof}
Combining (\ref{DqProp}) with (\ref{SDC1}), and then using (\ref{cadrule})
we get%
\begin{equation*}
q^{\ell }{\mathcal{R}}_{\ell ,n}(x)\mathbb{D}_{q^{-\ell }}^{2}(\mathbb{U}%
_{n}^{(a)})(q^{\ell }x;q)+{\mathcal{S}}_{\ell ,n}(x){\mathscr D}_{q^{-\ell }}%
\mathbb{U}_{n}^{(a)}(q^{\ell }x;q)+{\mathcal{T}}_{\ell ,n}(x)\mathbb{U}%
_{n}^{(a)}(x;q)=0.
\end{equation*}%
Next, replacing $x$ by $q^{-\ell }x$ and using (\ref{DqProOK}), yields%
\begin{equation*}
{\mathcal{R}}_{\ell ,n}(q^{-\ell }x)\mathbb{D}_{q^{\ell }}^{2}\mathbb{U}%
_{n}^{(a)}(x;q)+{\mathcal{S}}_{\ell ,n}(q^{-\ell }x){\mathscr D}_{q^{-\ell }}%
\mathbb{U}_{n}^{(a)}(x;q)+{\mathcal{T}}_{\ell ,n}(q^{-\ell }x)\mathbb{U}%
_{n}^{(a)}(q^{-\ell }x;q)=0,
\end{equation*}%
which is (\ref{SDC2}). This completes the proof.
\end{proof}

%%%%%%%%%%%%%%%%%%%%%%%%%%%%%%%%%%%%%%%%%%%%%%%%%%%%%%%%%%%%%%%%%%%%%%%%%%%%%%%%

%%%%%%%%%%%%%%%%%%%%%%%%%%%%%%%%%%%%%%%%%%%%%%%%%%%%%%%%%%%%%%%%%%%%%%%%%%%%%%%%

%%%%%%%%%%%%%%%%%%%%%%%%%%%%%%%%%%%%%%%%%%%%%%%%%%%%%%%%%%%%%%%%%%%%%%%%%%%%%%%%

%%%%%%%%%%%%%%%%%%%%%%%%%%%%%%%%%%%%%%%%%%%%%%%%%%%%%%%%%%%%%%%%%%%%%%%%%%%%%%%%

\section{Jacobi Fractions and Al-Salam--Carlitz I-Sobolev type polynomials}

%%%%%%%%%%%%%%%%%%%%%%%%%%%%%%%%%%%%%%%%%%%%%%%%%%%%%%%%%%%%%%%%%%%%%%%%%%%%%%%%

%%%%%%%%%%%%%%%%%%%%%%%%%%%%%%%%%%%%%%%%%%%%%%%%%%%%%%%%%%%%%%%%%%%%%%%%%%%%%%%%

In this section, we state some fresh results relating the elements in the family of orthogonal polynomials $\{\mathbb{U}_{n}^{(a)}(x;q)\}_{n\geq 0}$ and Jacobi fractions. Basic concepts and the main properties related to continued fractions, in particular with Jacobi fractions, and orthogonal polynomials can be found in~\cite{Chi-78}, Chapter 3, for instance. In this section, we use our original result (30) to consider $J-$fractions in a more general sense than in the previous text.

Given two sequences of polynomials with complex coefficients $%
\{a_n(x)\}_{n\ge1}$ and $\{b_n(x)\}_{n\ge0}$, the $J-$fraction associated to
the previous sequences is the formal expression 
\begin{equation}  \label{e1333}
b_0(x)+\frac{a_1(x)}{b_1(x)+\frac{a_2(x)}{b_2(x)+\frac{a_3(x)}{b_3(x)+\cdots}%
}}.
\end{equation}
For all $n\ge0$, the n-th convergent associated to the previous $J-$fraction
is given by 
\begin{equation*}
b_0(x) + \frac{a_1(x) \mid}{ \mid b_1(x)}+\cdots+\frac{a_n(x) \mid}{ \mid
b_n(x)}:=b_0(x)+\frac{a_1(x)}{b_1(x)+\frac{a_2(x)}{b_2(x)+\cdots+\frac{a_n(x)%
}{b_n(x)}}}.
\end{equation*}
The formal continued fraction (\ref{e1333}) is usually denoted by 
\begin{equation*}
b_0(x) + \frac{a_1(x) \mid}{ \mid b_1(x)}+\cdots+\frac{a_n(x) \mid}{ \mid
b_n(x)}+\cdots
\end{equation*}

%%%%%%%%%%%%%%%%%%%%%%%%%%%%%%%

\begin{proposition}
Let $\{\mathbb{U}_{n}^{(a)}(x;q)\}_{n\geq 0}$ be the sequence of
Al-Salam--Carlitz I-Sobolev type polynomials defined by \eqref{ASPTSHR}.
Then, $\{\mathbb{U}_{n}^{(a)}(x;q)\}_{n\geq 0}$ is the sequence of
denominators of the sequence of n-th convergents of the $J$-fraction  
\begin{equation}  \label{e1346}
\hat{\beta}_{\ell,0}(x)+\frac{\hat{\gamma}_{\ell,1}(x)\mid }{\mid \hat{\beta}%
_{\ell,1}(x)}+\frac{\hat{\gamma}_{\ell,2}(x)\mid }{\mid \hat{\beta}%
_{\ell,2}(x)}+\cdots+\frac{\hat{\gamma}_{\ell,n}(x)\mid }{\mid \hat{\beta}%
_{\ell,n}(x)}+\cdots ,
\end{equation}
where $\hat{\beta}_{\ell,n}(x)=\beta_{\ell,n-1}/\alpha_{\ell,n-1}$ and $\hat{%
\gamma}_{\ell,n}(x)=\gamma_{\ell,n-1}/\alpha_{\ell,n-1}$ for all $n\ge1$,
and any fixed $\hat{\beta}_{0,n}(x)\in\mathbb{C}[x]$.
\end{proposition}

\begin{proof}
We recall from (\ref{3TRR-RC}) that the sequence of Al-Salam--Carlitz
I-Sobolev type polynomials defined by \eqref{ASPTSHR} satisfy the following
recurrence formula:  
\begin{equation*}
\mathbb{U}_{n}^{(a)}(x;q)=\hat{\beta}_{\ell,n}(x)\mathbb{U}_{n-1}^{(a)}(x;q)+%
\hat{\gamma}_{\ell,n}(x)\mathbb{U}_{n-2}^{(a)}(x;q),\quad\mathbb{U}%
_{-1}^{(a)}(x;q)=0,\quad\quad\mathbb{U}_{0}^{(a)}(x;q)=1,\quad n\geq 1,
\end{equation*}
where $\hat{\beta}_{\ell,n}(x)$ and $\hat{\gamma}_{\ell,n}(x)$ are defined
as in the statements of the result. We define the sequence of polynomials $\{%
\mathcal{N}_{n}^{(a)}(x)\}_{n\ge-2}$ by the shifted recursion  
\begin{equation*}
{\mathcal{N}}_n^{(a)}(x;q)=\hat{\beta}_{\ell,n}(x){\mathcal{N}}%
_{n-1}^{(a)}(x;q)+\hat{\gamma}_{\ell,n}(x){\mathcal{N}}_{n-2}^{(a)}(x;q),%
\quad {\mathcal{N}}_{-2}^{(a)}(x;q)=0,\quad {\mathcal{N}}_{-1}^{(a)}(x;q)=1,%
\quad n\geq 0.
\end{equation*}
Consequently, one has that  
\begin{equation*}
{\mathcal{N}}_{0}^{(a)}(x;q)=\hat{\beta}_{\ell,0}(x),\quad\mathbb{U}%
_{0}^{(a)}(x;q)=1,
\end{equation*}
Also, for $n=1$ we have  
\begin{equation*}
{\mathcal{N}}_{1}^{(a)}(x;q)=\hat{\beta}_{\ell,0}(x)\hat{\beta}_{\ell,1}(x)+%
\hat{\gamma}_{\ell,1}(x),\quad\mathbb{U}_{1}^{(a)}(x;q)=\hat{\beta}%
_{\ell,1}(x),
\end{equation*}
which entail that  
\begin{equation*}
\frac{{\mathcal{N}}_{1}^{(a)}(x;q)}{\mathbb{U}_{1}^{(a)}(x;q)}=\hat{\beta}%
_{\ell,0}(x)+\frac{\hat{\gamma}_{\ell,1}(x)}{\hat{\beta}_{\ell,1}(x)}.
\end{equation*}
Let $n=2$. Then, it holds that  
\begin{equation*}
{\mathcal{N}}_{2}^{(a)}(x;q)=\hat{\beta}_{\ell,2}(x){\mathcal{N}}%
_{1}^{(a)}(x;q)+\hat{\gamma}_{\ell,2}(x)\hat{\beta}_{\ell,2}(x){\mathcal{N}}%
_{0}^{(a)}(x;q)=\hat{\beta}_{\ell,0}(x)[\hat{\beta}_{\ell,1}(x)\hat{\beta}%
_{\ell,2}(x)+\hat{\gamma}_{\ell,2}(x)]+\hat{\beta}_{\ell,2}(x)\hat{\gamma}%
_{\ell,1}(x),
\end{equation*}
and  
\begin{equation*}
\quad\mathbb{U}_{2}^{(a)}(x;q)=\hat{\beta}_{\ell,2}(x)\mathbb{U}%
_{1}^{(a)}(x;q)+\hat{\gamma}_{\ell,2}(x)\mathbb{U}_{0}^{(a)}(x;q)=\hat{\beta}%
_{\ell,1}(x)\hat{\beta}_{\ell,2}(x)+\hat{\gamma}_{\ell,2}(x).
\end{equation*}
Therefore, we derive  
\begin{equation*}
\frac{{\mathcal{N}}_{2}^{(a)}(x;q)}{\mathbb{U}_{2}^{(a)}(x;q)}=\hat{\beta}%
_{\ell,0}(x)+\frac{\hat{\gamma}_{\ell,1}(x)}{\hat{\beta}_{\ell,1}(x)+%
\displaystyle\frac{\hat{\gamma}_{\ell,2}(x)}{\hat{\beta}_{\ell,2}(x)}}.
\end{equation*}
A recursion argument determines the n-th convergent of the J-fraction in (%
\ref{e1346}), which coincides with $\frac{{\mathcal{N}}_{n}^{(a)}(x;q)}{%
\mathbb{U}_{n}^{(a)}(x;q)}$, and the proof can conclude. 
\end{proof}

An analogous result can be stated relating the denominators of the
convergents of a J-fraction and the sequence of Al-Salam--Carlitz I-Sobolev
type polynomials.

\begin{proposition}
Let $\{\mathbb{U}_{n}^{(a)}(x;q)\}_{n\geq 0}$ be the sequence of
Al-Salam--Carlitz I-Sobolev type polynomials defined by \eqref{ASPTSHR}.
Then, the polynomial $\mathbb{U}_{n+1}^{(a)}(x;q)$ is the numerator of the
n-th convergent of the $J$-fraction  
\begin{equation*}
\tilde{\beta}_{\ell,0}(x)+\frac{\tilde{\gamma}_{\ell,1}(x)\mid }{\mid \tilde{%
\beta}_{\ell,1}(x)}+\frac{\tilde{\gamma}_{\ell,2}(x)\mid }{\mid \tilde{\beta}%
_{\ell,2}(x)}+\cdots+\frac{\tilde{\gamma}_{\ell,n}(x)\mid }{\mid \tilde{\beta%
}_{\ell,n}(x)}+\cdots ,
\end{equation*}
for all $n\ge0$. Here, $\tilde{\beta}_{\ell,n}(x)=\hat{\beta}_{\ell,n+1}(x)=%
\frac{\beta_{\ell,n}(x)}{\alpha_{\ell,n}(x)}$ and $\tilde{\gamma}%
_{\ell,n}(x)=\hat{\gamma}_{\ell,n+1}(x)=\frac{\gamma_{\ell,n}(x)}{%
\alpha_{\ell,n}(x)}$, for all $n\ge 0$.
\end{proposition}

\begin{proof}
The recursion (\ref{3TRR-RC}) can be written in the form  
\begin{equation*}
\mathbb{U}_{n+1}^{(a)}(x;q)=\tilde{\beta}_{\ell,n}(x)\mathbb{U}%
_{n}^{(a)}(x;q)+\tilde{\gamma}_{\ell,n}(x)\mathbb{U}_{n-1}^{(a)}(x;q),\quad%
\mathbb{U}_{-1}^{(a)}(x;q)=0,\quad\quad\mathbb{U}_{0}^{(a)}(x;q)=1,\quad
n\geq 0.
\end{equation*}
We define the sequence $\{\mathcal{M}_{n}^{(a)}\}_{n\ge-1}$ by  
\begin{equation*}
{\mathcal{M}}_n^{(a)}(x;q)=\tilde{\beta}_{\ell,n}(x){\mathcal{M}}%
_{n-1}^{(a)}(x;q)+\tilde{\gamma}_{\ell,n}(x){\mathcal{M}}_{n-2}^{(a)}(x;q),%
\quad {\mathcal{M}}_{-1}^{(a)}(x;q)=0,\quad {\mathcal{M}}_{0}^{(a)}(x;q)=1,%
\quad n\geq 1.
\end{equation*}
Consequently, for $n=0$ we have  
\begin{equation*}
\quad\mathbb{U}_{1}^{(a)}(x;q)=\tilde{\beta}_{\ell,0}(x)\quad {\mathcal{M}}%
_{0}^{(a)}(x;q)=1,
\end{equation*}
whereas for $n=1$ we have  
\begin{equation*}
\quad\mathbb{U}_{2}^{(a)}(x;q)=\tilde{\beta}_{\ell,0}(x)\tilde{\beta}%
_{\ell,1}(x)+\tilde{\gamma}_{\ell,1}(x),\quad {\mathcal{M}}_{1}^{(a)}(x;q)=%
\tilde{\beta}_{\ell,1}(x).
\end{equation*}
Therefore, one can write the quotient  
\begin{equation*}
\frac{\quad\mathbb{U}_{2}^{(a)}(x;q)}{{\mathcal{M}}_{1}^{(a)}(x;q)}=\tilde{%
\beta}_{\ell,0}(x)+\frac{\tilde{\gamma}_{\ell,1}(x)}{\tilde{\beta}%
_{\ell,1}(x)}.
\end{equation*}
For $n=2$ we have  
\begin{equation*}
\mathbb{U}_{3}^{(a)}(x;q)=\tilde{\beta}_{\ell,0}(x)[\tilde{\beta}_{\ell,1}(x)%
\tilde{\beta}_{\ell,2}(x)+\tilde{\gamma}_{\ell,2}(x)]+\tilde{\beta}%
_{\ell,2}(x)\tilde{\gamma}_{\ell,1}(x),\quad {\mathcal{M}}_{2}^{(a)}(x;q)=%
\tilde{\beta}_{\ell,1}(x)\tilde{\beta}_{\ell,2}(x)+\tilde{\gamma}%
_{\ell,2}(x).
\end{equation*}
Thus,  
\begin{equation*}
\frac{\mathbb{U}_{3}^{(a)}(x;q)}{{\mathcal{M}}_{2}^{(a)}(x;q)}=\tilde{\beta}%
_{\ell,0}(x)+\frac{\tilde{\gamma}_{\ell,1}(x)}{\tilde{\beta}_{\ell,1}(x)+%
\displaystyle\frac{\tilde{\gamma}_{\ell,2}(x)}{\tilde{\beta}_{\ell,2}(x)}}.
\end{equation*}
A recursion argument allows to conclude the result.
\end{proof}

%%%%%%%%%%%%%%%%%%%%%%%%%%%%%%%

Following a similar argument, the next result holds.

\begin{corollary}
Let $n\ge1$. Then, for $\ell=-1,1$ one has 
%$$\frac{\mathbb{U}_{n+1}^{(a)}(x;q)}{\mathbb{U}_{n}^{(a)}(x;q)}=\frac{\mathbb{U}_{1}^{(a)}(x;q)}{\tilde{\beta}_{\ell,0}\tilde{\beta}_{\ell,1}+\tilde{\gamma}_{\ell,1}}\prod_{i=1}^{n}\tilde{\beta}_{\ell,i}+\sum_{i=0}^{n}\tilde{\gamma}_{\ell,i}\tilde{\beta}_{\ell,i+1}\tilde{\beta}_{\ell,i+2}\cdots \tilde{\beta}_{\ell,n},$$
\begin{equation*}
\frac{\mathbb{U}_{n+1}^{(a)}(x;q)}{\mathbb{U}_{n}^{(a)}(x;q)}=\sum_{i=1}^{n} 
\tilde{\gamma}_{\ell,i}(x) \prod_{h=i+1}^{n}\tilde{\beta}_{\ell,h}(x) ,
\end{equation*}
with the last term in the previous sum being reduced to $\tilde{\gamma}%
_{\ell,n}$.
\end{corollary}

\begin{proof}
It can be derived from the fact that $\omega_n=\frac{\mathbb{U}%
_{n+1}^{(a)}(x;q)}{\mathbb{U}_{n}^{(a)}(x;q)}$ satisfies that $\omega_{n+1}=%
\tilde{\beta}_{\ell,n+1}(x)+\tilde{\gamma}_{\ell,n+1}(x)/\omega_n$, for all $%
n\ge 0$, and a recursion argument.
\end{proof}

\section{Examples and further comments}

In this section, we illustrate the theory with some explicit first elements
in $\{\mathbb{U}_{n}^{(a)}(x;q,j)\}_{n\ge0}$.

Let $j=2$. We have 
\begin{equation*}
\mathbb{U}_{0}^{(a)}(x;q,2)=1,\quad \mathbb{U}_{1}^{(a)}(x;q,2)= x-a-1,
\end{equation*}
\begin{equation*}
\mathbb{U}_{2}^{(a)}(x;q,2)=x^2+(-aq-a-q-1)x+a^2q+aq+a+q,
\end{equation*}
\begin{equation*}
\mathbb{U}_{3}^{(a)}(x;q,2)=x^3+a_2x^2+a_1x+a_0,
\end{equation*}
where 
\begin{align*}
a_{2}&=-\frac{Z_q{a}^{3}{q}^{4}\left( q;q\right) _{2}+Z_q{a}^{2}{q}%
^{4}\left( q;q\right) _{2}-Z_q {a}^{3}q\left( q;q\right) _{2}-a\lambda\,{q}%
^{4}}{Z_q{a}^{2}{q}^{2}\left( q;q\right) _{2}-Z_q{a}^{2}q\left( q;q\right)
_{2}-\lambda\,{q}^{2}-\mu\,{q}^{2}-2 \,\lambda\,q-2\,\mu\,q-\lambda-\mu} \\
&-\frac{-Z_q{a}^{2}q\left( q;q\right) _{2}-3\,a \lambda\,{q}^{3}-\mu\,{q}%
^{4}-4\,a\lambda\,{q}^{2}-3\,\mu\,{q}^{3} -3\, a\lambda\,q-4\,\mu\,{q}%
^{2}-a\lambda-3\,\mu\,q-\mu}{Z_q{a}^{2}{q}^{2}\left( q;q\right) _{2}-Z_q{a}%
^{2}q\left( q;q\right) _{2}-\lambda\,{q}^{2}-\mu\,{q}^{2}-2
\,\lambda\,q-2\,\mu\,q-\lambda-\mu},
\end{align*}
\begin{align*}
a_1&=\frac {Z_q{a}^{4}{q}^{5}\left( q;q\right) _{2}+Z_q{a}^{3}{q}^{5}\left(
q;q\right) _{2}+Z_q{\ a}^{3}{q}^{4}\left( q;q\right) _{2}+Z_q{a}^{2}{q}%
^{5}\left( q;q\right) _{2}-Z_q{a}^{4}{q}^ {2}\left( q;q\right) _{2}-{a}%
^{2}\lambda\,{q}^{5}}{Z_q{a}^{2}{q}^{2}\left( q;q\right) _{2}-Z_q{a}%
^{2}\left( q;q\right) _{2}-\lambda\,{q}^{2}-\mu\,{q}^{2}-2\,
\lambda\,q-2\,\mu\,q-\lambda-\mu} \\
&\frac{-Z_q{a}^{3}{q}^{2}\left( q;q\right) _{2}-3\,{a}^{2}\lambda\,{q}^{4}+{a%
}^{2}\mu\,{q}^{4}-Z_q{a}^{3}q\left( q;q\right) _{2}-Z_q{a}^{2}{q}^{2}\left(
q;q\right) _{2}-4\,{a}^{2}\lambda\,{q}^{3}+3\,{a}^{2} \mu\,{q}^{3}-\mu\,{q}%
^{5}}{Z_q{a}^{2}{q}^{2}\left( q;q\right) _{2}-Z_q{a}^{2}\left( q;q\right)
_{2}-\lambda\,{q}^{2}-\mu\,{q}^{2}-2\, \lambda\,q-2\,\mu\,q-\lambda-\mu} \\
&\frac{-3\,{a}^{2}\lambda\,{q}^{2}+4\,{a}^{2}\mu\,{q }^{2}+\lambda\,{q}%
^{4}-3\,\mu\,{q}^{4}-{a}^{2}\lambda\,q+3\,{a}^{2}\mu \,q+3\,\lambda\,{q}%
^{3}-4\,\mu\,{q}^{3}+{a}^{2}\mu+4\,\lambda\,{q}^{2} }{Z_q{a}^{2}{q}%
^{2}\left( q;q\right) _{2}-Z_q{a}^{2}\left( q;q\right) _{2}-\lambda\,{q}%
^{2}-\mu\,{q}^{2}-2\, \lambda\,q-2\,\mu\,q-\lambda-\mu} \\
&\frac{-3\,\mu\,{q}^{2}+3\,\lambda\,q-\mu\,q+\lambda}{Z_q{a}^{2}{q}%
^{2}\left( q;q\right) _{2}-Z_q{a}^{2}\left( q;q\right) _{2}-\lambda\,{q}%
^{2}-\mu\,{q}^{2}-2\, \lambda\,q-2\,\mu\,q-\lambda-\mu},
\end{align*}
and 
\begin{align*}
a_0&=-\frac{Z_q{a}^{5}{q}^{5}\left( q;q\right) _{2}-Z_q{a}^{5}{q}^{4}\left(
q;q\right) _{2}+Z_q {a}^{4}{q}^{5}\left( q;q\right) _{2}+Z_q{a}^{3}{q}%
^{5}\left( q;q\right) _{2}+Z_q{a}^{2}{q} ^{5}\left( q;q\right) _{2}-{a}%
^{3}\lambda\,{q}^{5}}{Z{a}^ {2}{q}^{2}\left( q;q\right) _{2}-Z_q{a}%
^{2}q\left( q;q\right) _{2}-\lambda\,{q}^{2}-\mu \,{q}^{2}-2\,\lambda\,q-2\,%
\mu\,q-\lambda-\mu} \\
&-\frac{-Z_q{a}^{4}{q}^{2}\left( q;q\right) _{2}-Z_q{a}^{2}{q}^{4}\left(
q;q\right) _{2}-2\,{a}^{3}\lambda\,{q}^{4}+{a}^{3}\mu \,{q}^{4}-Z_q{a}^{3}{q}%
^{2}\left( q;q\right) _{2}-{a}^{3}\lambda\,{q}^{3}+3\,{a} ^{3}\mu\,{q}^{3}}{%
Z_q{a}^ {2}{q}^{2}\left( q;q\right) _{2}-Z_q{a}^{2}q\left( q;q\right)
_{2}-\lambda\,{q}^{2}-\mu \,{q}^{2}-2\,\lambda\,q-2\,\mu\,q-\lambda-\mu} \\
&-\frac{+{a}^{2}\mu\,{q}^{4}+3\,{a}^{3}\mu\,{q}^{2}+3\,{a}^{2} \mu\,{q}%
^{3}+a\lambda\,{q}^{4}-\mu\,{q}^{5}+{a}^{3}\mu\,q+4\,{a}^{2} \mu\,{q}%
^{2}+3\,a\lambda\,{q}^{3}+\lambda\,{q}^{4}}{Z_q{a}^ {2}{q}^{2}\left(
q;q\right) _{2}-Z_q{a}^{2}q\left( q;q\right) _{2}-\lambda\,{q}^{2}-\mu \,{q}%
^{2}-2\,\lambda\,q-2\,\mu\,q-\lambda-\mu} \\
&-\frac{-2\,\mu\,{q}^{4}+3\, {a}^{2}\mu\,q+4\,a\lambda\,{q}^{2}+3\,\lambda\,{%
q}^{3}-\mu\,{q}^{3}+{a }^{2}\mu+3\,a\lambda\,q+3\,\lambda\,{q}%
^{2}+a\lambda+\lambda\,q}{Z_q{a}^ {2}{q}^{2}\left( q;q\right) _{2}-Z_q{a}%
^{2}q\left( q;q\right) _{2}-\lambda\,{q}^{2}-\mu \,{q}^{2}-2\,\lambda\,q-2\,%
\mu\,q-\lambda-\mu},
\end{align*}
with $Z_q=\left( q,a,a^{-1}q;q\right) _{\infty }$.

Let $j=3$. We have 
\begin{equation*}
\mathbb{U}_{0}^{(a)}(x;q,3)=1,\quad \mathbb{U}_{1}^{(a)}(x;q,2)= x-a-1,
\end{equation*}
\begin{equation*}
\mathbb{U}_{2}^{(a)}(x;q,3)=x^2+(-aq-a-q-1)x+a^2q+aq+a+q,
\end{equation*}
\begin{align*}
\mathbb{U}%
_{3}^{(a)}(x;q,3)&=x^3+(-aq^2-aq-q^2-a-q-1)x^2+(a^2q^3+a^2q^2+aq^3+a^2q+2aq^2
\\
&q^3+2aq+q^2+a+q)x-a^3q^3-a^2q^3-a^2q^2-aq^3-a^2q-aq^2-q^3-aq,
\end{align*}

We observe that in the case of $j=2,3$, it holds that the nature of the
polynomials is much simpler for $j>n$. This is due to the definition of the
polynomials. Indeed, observe that 
\begin{equation*}
\mathcal{D}^{j}_qx^{n}(x;q)\equiv 0,\qquad \hbox{ for }j>n.
\end{equation*}
This can be proved directly from the definition of the differential operator 
$\mathcal{D}_q$ applied on any monomial. Therefore, given $j\ge 1$, one has
that $\mathbb{U}_{n}^{(a)}(x;q,j)$ coincides with $U_{n}^{(a)}(x;q)$ for all 
$0\le n<j$. The values of $\lambda,\mu$ are irrelevant for these first
polynomials in the sequence. This phenomenon can also be observed in
different points through the work. More precisely, the representation (\ref%
{e744}) of $\mathbb{U}_{n}^{(a)}(x;q)$ in terms of $U_n^{(a)}(x;q)$ is made
in terms of the quantities $\Delta_{j,n}^{(i)}(a)$, for $i=1,2$. The
definition of these two elements is made in terms of the determinant of a
matrix with a null column for $j>n$. Therefore, formula (\ref{e744}) states
the coincidence of Al-Salam--Carlitz I polynomials and the Sobolev type
polynomials. This property is also directly observed at the Fourier
coefficients $a_{n,k}$.

We also remark that the choice of $\lambda=\mu=0$ provides $\mathbb{U}%
_{n}^{(a)}(x;q,j)=U_n^{(a)}(x;q)$ for every $n$. We observe this is the case
for the polynomial $\mathbb{U}_{3}^{(a)}(x;q,2)$, which is given by 
\begin{align*}
\mathbb{U}_{3}^{(a)}(x;q,2)&={x}^{3}+(-a{q}^{2}-aq-{q}^{2}-a-q-1){x} ^{2}+({a%
}^{2}{q}^{3}+{a}^{2}{q}^{2}+a{q}^{3} \\
&+{a}^{2}q+2\,a{q}^{2}+ {q}^{3}+2\,aq+{q}^{2}+a+q)x-{a}^{3}{q}^{3}-{a}^{2}{q}%
^{3}-{a}^{2}{q }^{2}-a{q}^{3}-{a}^{2}q-a{q}^{2}-{q}^{3}-aq,
\end{align*}
after evaluation at $\lambda=\mu=0$.

\begin{figure}[tbp]
\centering
\includegraphics[width=0.45\textwidth]{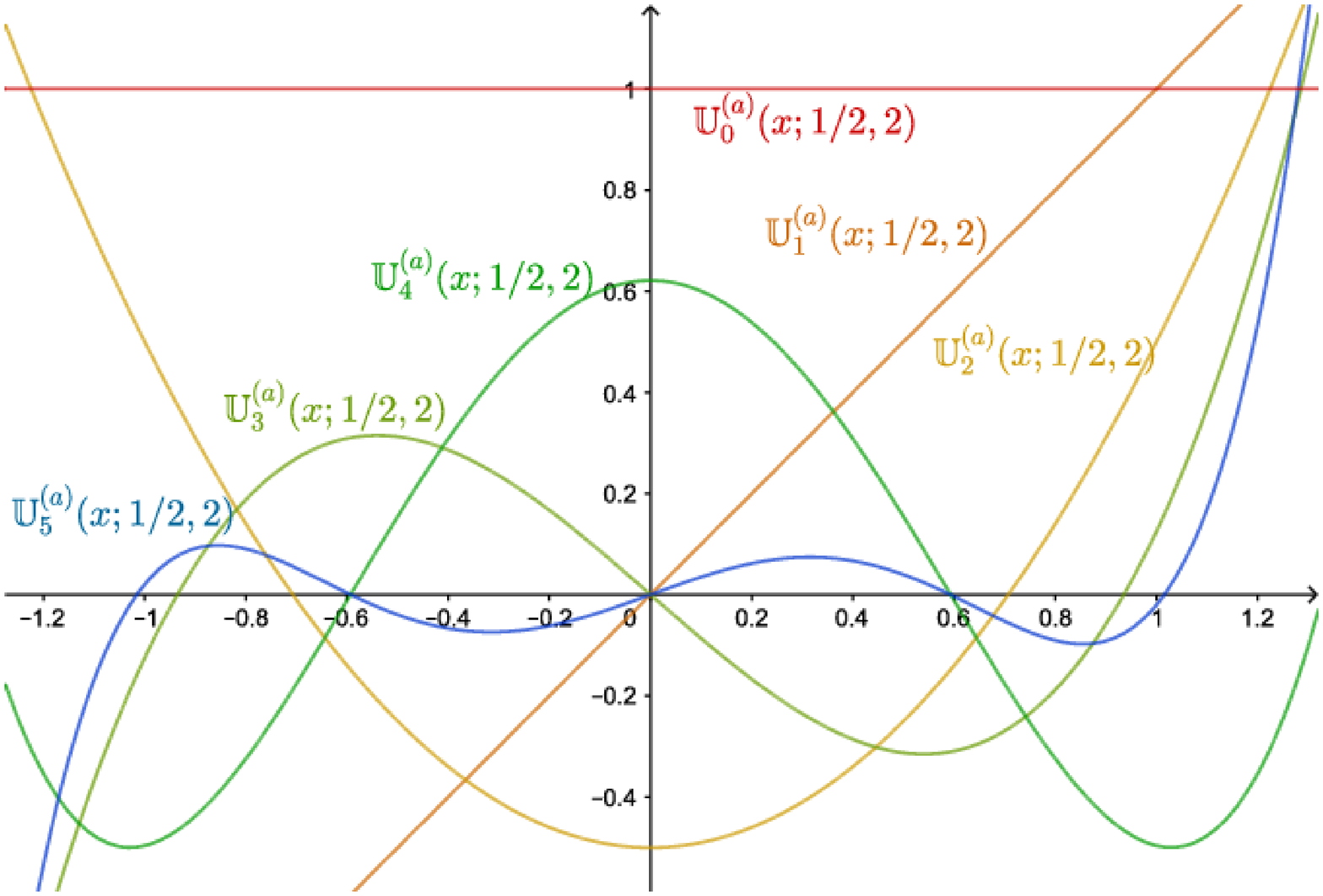} \includegraphics[width=0.45%
\textwidth]{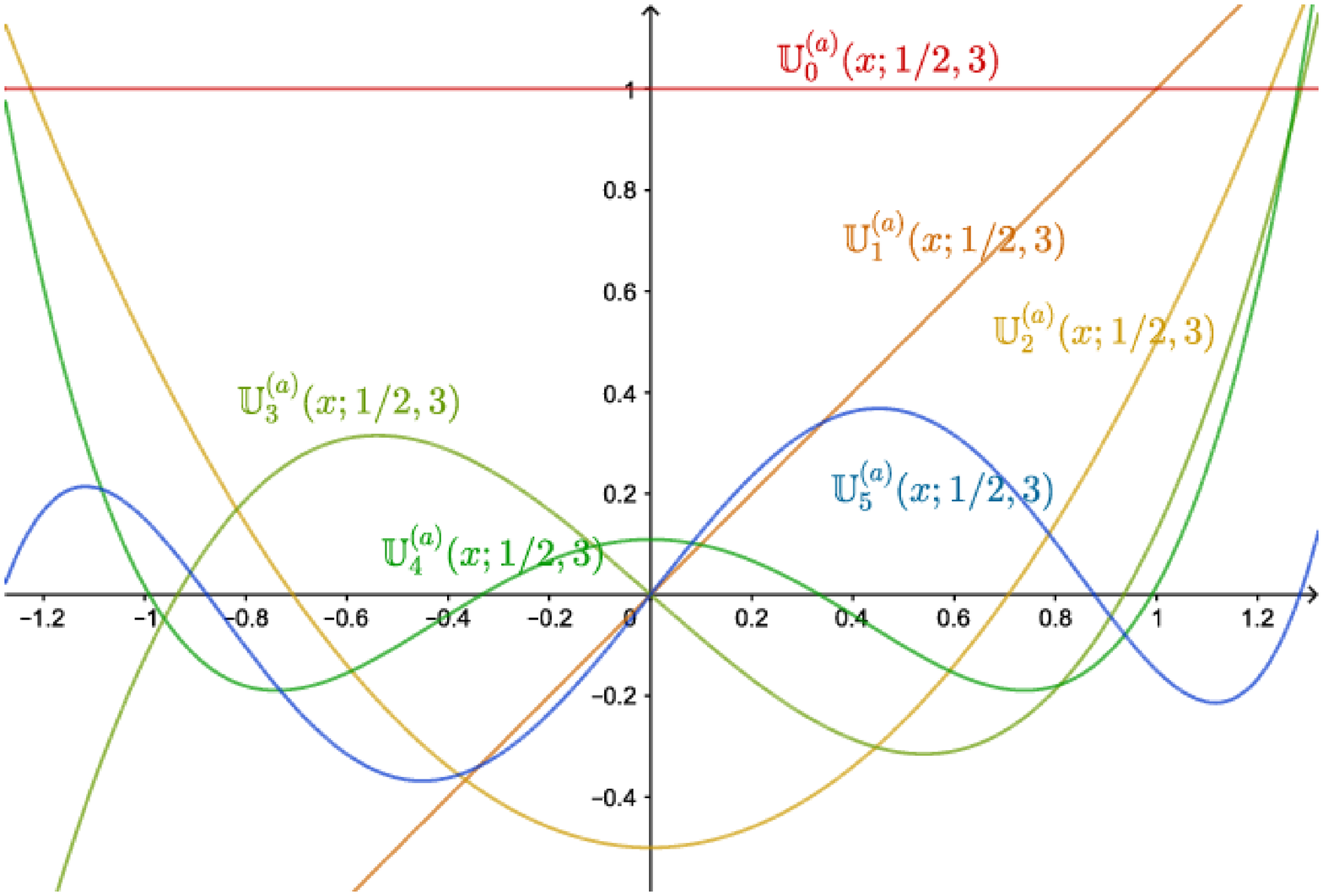} 
\caption{$\mathbb{U}_{n}^{(-1)}(x;\frac{1}{2},j)$ for $j=2$ (left) and $j=3$
(right) for $n=0,1,\ldots,5$}
\end{figure}

%%%%%%%%%%%%%%%%%%%%%%%%%%%%%%%%%%%%%%%%%%%%%%%%%%%%%%%%%%%%%%%%%%%%%%%%%%%%%%%%%%%%%%%%%%%%%%%%%%%%%%%

%%%%%%%%%%%%%%%%%%%%%%%%%%%%%%%%%%%%%%%%%%%%%%%%%%%%%%%%%%%%%%%%%%%%%%%%%%%%%%%%%%%%%%%%%%%%%%%%%%%%%%%

%%%%%%%%%%%%%%%%%%%%%%%%%%%%%%%%%%%%%%%%%%%%%%%%%%%%%%%%%%%%%%%%%%%%%%%%%%%%%%%%%%%%%%%%%%%%%%%%%%%%%%%

%%%%%%%%%%%%%%%%%%%%%%%%%%%%%%%%%%%%%%%%%%%%%%%%%%%%%%%%%%%%%%%%%%%%%%%%%%%%%%%%%%%%%%%%%%%%%%%%%%%%%%%

\section{Conclusions and open problems}

\label{S6-ConclOP}

%%%%%%%%%%%%%%%%%%%%%%%%%%%%%%%%%%%%%%%%%%%%%%%%%%%%%%%%%%%%%%%%%%%%%%%%%%%%%%%%%%%%%%%%%%%%%%%%%%%%%%%

%%%%%%%%%%%%%%%%%%%%%%%%%%%%%%%%%%%%%%%%%%%%%%%%%%%%%%%%%%%%%%%%%%%%%%%%%%%%%%%%%%%%%%%%%%%%%%%%%%%%%%%

Using well known techniques, we provide up to four different versions of the
second order $q$-difference equations satisfied by the monic polynomials $\{%
\mathbb{U}_{n}^{(a)}(x;q,j)\}_{n\geq 0}$, orthogonal with respect to a
Sobolev-type inner product associated to the Al-Salam--Carlitz I orthogonal
polynomials. The studied inner product involves an arbitrary number of $q$%
-derivatives, evaluated on the two boundaries of the orthogonality interval
of the Al-Salam--Carlitz I orthogonal polynomials. We gave two
representations for $\mathbb{U}_{n}^{(a)}(x;q,j)$, one as a linear
combination of two consecutive Al-Salam--Carlitz I orthogonal polynomials,
and other as a $_{3}\phi _{2}$ series. We state not only as usual, but two
different versions of structure relations, which lead to the corresponding
ladder operators, which help us to find up to four different versions of the
second order linear $q$-difference equation satisfied by $\mathbb{U}%
_{n}^{(a)}(x;q,j)$. Finally, as a truly original contribution to the
literature, we obtained a three term recurrence formula with rational
coefficients satisfied by $\mathbb{U}_{n}^{(a)}(x;q,j)$, which is the key
point to establish an appealing generalization of the so-called $J$%
-fractions to the framework of Sobolev-type orthogonality. As problems to be
addressed in a future contribution, we consider to analyze the effect of
having two mass points, each one on a different side of the bounded
orthogonality interval, in the parity of the corresponding Sobolev-type
orthogonal sequence. We also wish to carry out an in-depth analysis on the
zero behavior of these polynomials, as well as to study several of their
asymptotic properties.

%%%%%%%%%%%%%%%%%%%%%%%%%%%%%%%%%%%%%%%%%%%%%%%%%%%%%%%%%%%%%%%%%%%%%%%%%%%%%%%%%%%%%%%%%%%%%%%%%%%%%%%

%%%%%%%%%%%%%%%%%%%%%%%%%%%%%%%%%%%%%%%%%%%%%%%%%%%%%%%%%%%%%%%%%%%%%%%%%%%%%%%%%%%%%%%%%%%%%%%%%%%%%%%

\section*{Acknowledgments}

%%%%%%%%%%%%%%%%%%%%%%%%%%%%%%%%%%%%%%%%%%%%%%%%%%%%%%%%%%%%%%%%%%%%%%%%%%%%%%%%%%%%%%%%%%%%%%%%%%%%%%%

%%%%%%%%%%%%%%%%%%%%%%%%%%%%%%%%%%%%%%%%%%%%%%%%%%%%%%%%%%%%%%%%%%%%%%%%%%%%%%%%%%%%%%%%%%%%%%%%%%%%%%%

The work of the second (EJH) and third (AL) authors was funded by Dirección
General de Investigación e Innovación, Consejería de Educación e Investigació%
n of the Comunidad de Madrid (Spain), and Universidad de Alcalá under grant
CM/JIN/2019-010, Proyectos de I+D para Jóvenes Investigadores de la
Universidad de Alcalá 2019. The first author (CH) wishes to thank
Departamento de Física y Matemáticas de la Universidad de Alcalá for its
support. 
%=====================================
% References, variant B: internal bibliography
%=====================================


\begin{thebibliography}{99}

\bibitem{AsC-MN65} Al-Salam, W. A.; Carlitz, L. Some orthogonal $q$%
-polynomials. \emph{Math. Nachr.} \textbf{1965}, 30, 47--61.

\bibitem{AS-JEDA13} Arvesú, J.; Soria-Lorente, A. First-order
non-homogeneous $q$-difference equation for Stieltjes function
characterizing $q$-orthogonal polynomials. \emph{J. Difference Equ. Appl.} 
\textbf{2013}, 19(5), 814--838.

\bibitem{AI-SPM83} Askey, R.; Ismail, M. E. H. A generalization of
ultraspherical polynomials, in: P. Erdös (Ed.), \emph{Studies in Pure
Mathematics} \textbf{1983}, 55--78.

\bibitem{AS-LMP93} Askey, R.; Suslov, S. K. The $q$-harmonic oscillator and
the Al-Salam and Carlitz polynomials. \emph{Lett Math Phys.} \textbf{1993},
29, 123--132.

\bibitem{B-JCAM95} Bavinck, H. On polynomials orthogonal with respect to an
inner product involving differences. \emph{J. Comput. Appl. Math.} \textbf{%
1995}, 57, 17--27.

\bibitem{B-AA95} Bavinck, H. On polynomials orthogonal with respect to an
inner product involving differences (The general case). \emph{Appl. Anal.} 
\textbf{1995}, 59, 233--240.

\bibitem{B-IM96} Bavinck, H. A difference operator of infinite order with
the Sobolev-type Charlier polynomials as eigenfunctions. \emph{Indag. Math.} 
\textbf{1996}, 7(3), 281--291.

\bibitem{BK-JAT95} Bavinck, H.; Koekoek, R. On a difference equation for
generalizations of Charlier polynomials. \emph{J. Approx. Theory} \textbf{%
1995}, 81(2), 195--206.

\bibitem{Chi-78} Chihara, T. S. \emph{An Introduction to Orthogonal
Polynomials}; Math. Appl. Ser. 13; Gordon and Breach: New York, NY, USA,
1978.

\bibitem{CS-JDE18} Costas-Santos, R. S.; Soria-Lorente, A. Analytic
properties of some basic hypergeometric-Sobolev-type orthogonal polynomials. 
\emph{J. Difference Equ. Appl.} \textbf{2018}, 24(11), 1715--1733.

\bibitem{DA-JPA05} Doha, E.; Ahmed, H. Efficient algorithms for construction
of recurrence relations for the expansion and connection coefficients in
series of Al-Salam--Carlitz I polynomials. \emph{J. Phys. A: Mathematical
and General}\ \textbf{2005}, 38(47), 10107--10121.

\bibitem{E-PEAS09} Ernst, T. $q$-calculus as operational algebra. \emph{%
Proc. Est. Acad. Sci.} \textbf{2009}, 58(2), 73--97.

\bibitem{FMM-arXiv20} Filipuk, G.; Mañas-Mañas, Juan F.; Moreno-Balcázar,
Juan J. Ladders operators for general discrete Sobolev orthogonal
polynomials, arXiv:2006.14391 [math.CA] 2020.

\bibitem{GGH-M20} Garza, L. G.; Garza, L. E.; Huertas, E. J. On differential
equations associated with perturbations of orthogonal polynomials on the
unit circle. \emph{Mathematics} \textbf{2020}, 8(2), 246.

\bibitem{GR-Basic04} Gasper G.; Rahman M. \emph{Basic Hypergeometric Series}
; Encyclopedia of Mathematics and its Applications, Volume 96; Cambridge
University Press: Cambridge, UK, 2004.

\bibitem{H-MN49} Hahn, W. Über Orthogonalpolynome, die $q$%
-Differenzengleichungen genügen. \emph{Math. Nachr.} \textbf{1949}, 2, 4--34.

\bibitem{HS-NA19} Huertas, E. J.; Soria-Lorente, A. New analytic properties
of nonstandard Sobolev-type Charlier orthogonal polynomials. \emph{Numer.
Algorithms} \textbf{2019}, 82, 41--68.

\bibitem{Ismail-05} Ismail, M. E. H. \emph{Classical and Quantum Orthogonal
Polynomials in One Variable}; Encyclopedia of Mathematics and its
Applications, Volume 98; Cambridge University Press: Cambridge, UK, 2005.

\bibitem{IS-CJM88} Ismail, M. E. H.; Stanton, D. On the Askey-Wilson and
Rogers polynomials. \emph{Canad. J. Math.} \textbf{1988}, 40, 1025--1045.

\bibitem{koek2010} Koekoek, R.; Lesky, P. A.; Swarttouw R. F. \emph{%
Hypergeometric Orthogonal Polynomials and their q-analogues}. Springer
Science \& Business Media, 2010.

\bibitem{KS-JCAM01} Koepf, W.; Schmersau, D. On a structure formula for
classical q-orthogonal polynomials. \emph{J. Comput. Appl. Math.} \textbf{%
2001}, 136(1-2), 99--107.

\bibitem{MX-EM15} Marcellán, F.; Xu, Y. On Sobolev orthogonal polynomials. 
\emph{Expo. Math.} \textbf{2015}, 33(3), 308--352.

%\bibitem{M-JCAM01} Martínez-Finkelshtein, A. Analytic aspects of Sobolev orthogonality revisited, \emph{J. Comput. and Appl. Math.} \textbf{2001}, 127, 255--266.

\bibitem{NikUvSu-91} Nikiforov, A. F.; Uvarov, V. B.; Suslov, S. K. \emph{%
Classical Orthogonal Polynomials of a Discrete Variable}. Springer Series in
Computational Physics, Springer Verlag, Berlin, 1991.

\bibitem{R-AA20} Rebocho, M. N. On the second-order holonomic equation for
Sobolev-type orthogonal polynomials, \emph{Appl. Anal.} \textbf{2020}. In
press.

%\bibitem{S-JCAM79} Sánchez-Dehesa, J. On a general system of orthogonal q-polynomials. \emph{J. Comput. Appl. Math.} \textbf{1979}, 5, 37--45.

\bibitem{shang1} Shang, Y. A note on the commutativity of prime near-rings. 
\emph{Algebra Colloq.} \textbf{2015}, 22(3), 361--366.

\bibitem{shang2} Shang, Y. A study of derivations in prime near-rings. \emph{%
Math. Balkanica} \textbf{2011}, 25(4), 413--418.

\bibitem{srivastava2011zeta} Srivastava, H. M.; Choi, J. \emph{Zeta and
q-Zeta Functions and Associated Series and Integrals}. Elsevier, New York,
NY, USA, 2012.
\end{thebibliography}
\end{document}